\documentclass[11pt]{amsart}
\usepackage[T1]{fontenc}
\usepackage{babel}
\usepackage{amsmath}
\usepackage{amsthm}
\usepackage{amssymb}
\usepackage{color}
\usepackage{newclude}
\usepackage{amsmath, amssymb, amscd, mathrsfs, slashed, enumerate, url}
\usepackage{amsfonts}
\usepackage{color}
\usepackage{extsizes}
\usepackage{xcolor}
\usepackage[all,cmtip]{xy}
\usepackage{multicol}
\usepackage{enumitem} 
\usepackage{indentfirst}
\usepackage{latexsym}
\usepackage{bm}
\usepackage{graphicx}
\usepackage{subfigure,esint}
\usepackage{float}
\usepackage{verbatim}
\usepackage{comment}
\usepackage{esint}
\usepackage[top=1in, bottom=1in, left=1.10in, right=1.10in]{geometry}
\usepackage{epsfig,dsfont,amsthm,amsfonts,amsbsy}
\usepackage{overpic}
\usepackage{hyperref}

\newcommand{\Z}{\mathbb{Z}}
\newcommand{\N}{\mathbb{N}}
\newcommand{\R}{\mathbb{R}}
\newcommand{\C}{\mathbb{C}}

\newcommand{\F}{\mathcal{F}}
\newcommand{\G}{\mathcal{G}}
\newcommand{\ZT}{\mathbb{Z}/2}

\newcommand{\Tr}{\operatorname{Tr}}

\newcommand{\Sym}{\operatorname{Sym}}
\renewcommand{\d}{\operatorname{d}}
\newcommand{\del}{\partial}

\renewcommand{\O}{\mathcal{O}}

\renewcommand{\Re}{\operatorname{Re}}

\newcommand{\pa}{\partial}

\newcommand{\Msch}{\MS}

\newcommand{\MS}{\mathcal{S}}

\newtheorem{theorem}{Theorem}[section]
\newtheorem{proposition}[theorem]{Proposition}
\newtheorem{lemma}[theorem]{Lemma}
\newtheorem{remark}[theorem]{Remark}
\newtheorem{question}[theorem]{Question}

\newtheorem{corollary}[theorem]{Corollary}
\newtheorem{definition}[theorem]{Definition}
\newtheorem{example}[theorem]{Example}

\newcommand{\MI}{\mathcal{I}}

\newcommand{\mbR}{\mathbb{R}}

\title{Metric perturbations and deformations of \texorpdfstring{$\vec{k}$}{k}-nondegenerate $\ZT$-harmonic 1-forms}
\author{Siqi He and Willem Adriaan Salm}
\begin{document}
\begin{abstract}
	We study metric perturbations and deformation theory for degenerate \(\mathbb{Z}/2\)-harmonic \(1\)-forms.
	For a natural class of degenerate examples, we prove that after a suitable perturbation of the ambient Riemannian metric, the form can be deformed to a nearby non-degenerate \(\mathbb{Z}/2\)-harmonic \(1\)-form.
	Our argument combines analysis of the leading coefficients in the local expansion under metric perturbations with a quantitative Nash--Moser implicit function theorem.
\end{abstract}
\maketitle

\section{Introduction}

\(\mathbb{Z}/2\)-harmonic \(1\)-forms have emerged as a natural analytic object in compactification problems in gauge theory.
Starting from Taubes' analytic compactification of moduli spaces of flat \(\mathrm{SL}_2(\mathbb{C})\) connections on \(3\)-manifolds, these forms appear as ideal boundary data describing the limiting behaviour of sequences of connections; see \cite{taubes2013compactness,Taubescompactness18,Tau13a,walpuskizhang2019compactness}. 
They have since played a role in related analytic and geometric questions, providing a useful bridge between gauge theory, low-dimensional topology, and geometric analysis; see for instance \cite{BeraWalpuski2025Dirac,Chen2025Perturbation,Donaldson2021,EsfahaniLi2024Fueter,he23branched,He2025, HeParkerWalpuski26Deformation,HeWentworthZhang2024Z2, Parker2026Deformations, Salm24,takahashi2015moduli, Yan2025Nondegenerate}.

Let \((M,g)\) be a Riemannian manifold.
Let \(\Sigma\subset M\) be a co-oriented codimension-two submanifold, and let \(\mathcal{I}\) be a flat real line bundle over \(M\setminus \Sigma\) with monodromy \(-1\) along small loops linking \(\Sigma\).
A \(\mathbb{Z}/2\)-harmonic \(1\)-form is a triple \((\Sigma,\mathcal{I},\omega)\), where \(\omega\) is a smooth \(\mathcal{I}\)-valued harmonic \(1\)-form on \(M\setminus \Sigma\) such that \(|\omega|\) extends continuously across \(\Sigma\).
Equivalently, \(\omega\) may be viewed locally as a $\pm$-valued harmonic \(1\)-form with singular set \(\Sigma\), reflecting the nontrivial monodromy of \(\mathcal{I}\).

Assume \(\Sigma\) is smooth.
Results of \cite{Donaldson2021,Mazzeo1991} imply that \(\omega\) admits a local expansion. Namely, if $N$ is the normal bundle of $\Sigma$, then $\mathcal{I}$ determines a square root $N^{\frac{1}{2}}$ in the appropriate sense. Given a complex normal coordinate $z$ on $N$ and a local coordinate \(t\) on $\Sigma$, one has
\[
\omega = \d\Re\bigl(f_k(t)\, z^{k+\frac12}\bigr) + o\!\left(|z|^{k}\right),
\]
for some integer \(k\ge 1\) and a local section \(f_k\) of \(N^{-(k+\frac12)}\).

\begin{definition}[$\vec k$-nondegenerate]
	\label{def_k_nondegenerate}    Given a connected component $\Sigma_i$ of $\Sigma$, we say that \((\Sigma,\mathcal{I},\omega)\) is \emph{\(k\)-nondegenerate along $\Sigma_i$}, if \(k\geq 1\) is the smallest index for which \(f_k\not\equiv 0\) and \(f_k\) is nowhere vanishing along \(\Sigma_i\).
    Moreover, given a multi-index $\vec{k} = (k_1, \ldots, k_{b_0(\Sigma)})$, we say that  \((\Sigma,\mathcal{I},\omega)\) is \emph{\(\vec{k}\)-nondegenerate} if for each connected component $\Sigma_i \subset \Sigma$, the $\Z_2$-harmonic 1-form \((\Sigma,\mathcal{I},\omega)\) is \emph{\(k_i\)-nondegenerate} along $\Sigma_i$.
\end{definition}

The case \(\vec{k}=(1, \ldots, 1)\) is the classical non-degeneracy condition used in the literature (see for example \cite{Donaldson2021,Parker2026Deformations}), while \(k_i\geq 1\) describes a controlled finite-order degeneracy.
For classically non-degenerate \(\mathbb{Z}/2\)-harmonic \(1\)-forms, the resulting deformation theory leads to a well-posed local moduli problem and supports transversality and gluing arguments, with applications ranging from gauge-theoretic compactifications to calibrated geometry; see \cite{Donaldson2021,Parker2026Deformations,Yan2025Nondegenerate}.

From an analytic point of view, the non-degeneracy condition is crucial: It provides uniform control of the linearized deformation problem near the singular locus.
When non-degeneracy fails, this uniform control breaks down and the standard deformation theory no longer applies.
Despite the progress above, degeneracies are common and often unavoidable.
This paper addresses the following question.

\begin{question}\label{question_main}
	Given a degenerate \(\mathbb{Z}/2\)-harmonic \(1\)-form, can one perturb the ambient metric and deform the form to remove the degeneracy?
\end{question}

A basic obstruction is topological: the relevant leading coefficient naturally lives in a fractional power \(N^{-3/2}\) of the complex normal bundle \(N\to \Sigma\). If \(N^{-3/2}\) is nontrivial then any section must vanish somewhere, forcing degeneracy.
Thus, one cannot hope to avoid degenerate \(\mathbb{Z}/2\)-harmonic \(1\)-forms in general, and the deformation theory in the degenerate regime requires new ideas. 

In this paper we give an affirmative answer to Question~\ref{question_main} for \(\vec{k}\)-nondegenerate \(\mathbb{Z}/2\)-harmonic \(1\)-forms. Namely,

\begin{theorem}\label{thm:k-degenerate-theorem}
	Let \((\Sigma,\mathcal{I},\omega_0)\) be a \(\vec{k}\)-nondegenerate \(\mathbb{Z}/2\)-harmonic \(1\)-form on \((M,g_0)\), such that $N^{-\frac{3}{2}}$ is trivial.
	Then there exists a one-parameter family of metrics \(g_s\) extending \(g_0\), such that for each \(s\in(0,1)\) there exist non-degenerate \(\mathbb{Z}/2\)-harmonic \(1\)-forms \((\Sigma,\mathcal{I},\omega_s)\) with respect to \(g_s\), satisfying the following:
	\begin{itemize}
		\item [(i)] In the cohomology class of $L^2$-bounded $\Z/2$-harmonic 1-forms, \(\omega_0\) and \(\omega_s\) represent the same element.

		\item [(ii)] Given any compact set \(K\subset M\setminus \Sigma_0\), the family \(g_s\) can be chosen so that \(g_s|_K=g_0|_K\).
	\end{itemize}
\end{theorem}

The proof has two main steps.
First, after a small perturbation of the ambient metric, we deform a \(\vec k\)-nondegenerate form to one with improved leading behaviour (allowing a controlled additional singular term).
Second, we apply a refined version of the Nash--Moser implicit function theorem to carry out a Donaldson-type deformation argument for the coupled metric--form problem.

A main difficulty here is quantitative rather than conceptual. Nash--Moser techniques play a central role in the deformation theory of \(\mathbb{Z}/2\)-harmonic \(1\)-forms and have been used extensively in the literature, see for example \cite{Donaldson2021,DonaldsonLehmann2024,Parker2026Deformations, Salm24}.
In our setting, however, the standard Nash--Moser framework is not sufficiently precise. We need a finer version in which all constants are tracked at different small scales, so that the admissible solvability radius can be determined explicitly.

In our case, there are two main competing scales we need to keep track of. First, for a generic perturbation of $\Sigma$, the norm $|\omega|$ will not extend continuously over $\Sigma$. For the perturbation argument to succeed we need to measure the failure of this extension. Secondly, we need to estimate the size of the right inverse of the linearized operator. Because \((\Sigma,\mathcal{I},\omega_0)\) is degenerate, this estimate will worsen when $g_s$ converges to $g_0$.

Only after keeping track of this balance can one determine an explicit neighbourhood in the target space for which invertibility holds.
In our setting this requires tracking the dependence of all Nash--Moser constants on the small scale dictated by the vanishing order, which cannot be obtained from a black-box choice of weighted Sobolev spaces.
This is the reason we work with a Nash--Moser scheme, where all tame estimates are formulated with explicit control of the small parameter at each step. For this we give an explicit proof in Appendix \ref{sec:appendix}.

\textbf{Acknowledgements.} We wish to express our gratitude to many people for their interest and helpful comments. Among them are Jiahuang Chen, Andriy Haydys, Rafe Mazzeo, Greg Parker, Thomas Walpuski. S.H. is partially supported by NSFC grant No.12288201 and No.2023YFA1010500.

\section{\texorpdfstring{$\ZT$}{Z/2} harmonic 1-form over a closed manifold}
In this section we introduce the analytic foundations of \(\ZT\)-harmonic \(1\)-forms,
as developed in \cite{Donaldson2021,Tau13a} and refined in subsequent work. See also \cite{BeraWalpuski2025Dirac,MazzeoHaydysTakahashi} for further background and additional references. We first revisit \(L^2\)-bounded \(\ZT\)-harmonic \(1\)-forms and recall their \(L^2\)-Hodge theoretic interpretation via the canonical double branched covering. Secondly, we explain their local expansions and decay behaviour near the singular set. Finally, we finish this sections with a few examples.

\subsection{Hodge theory for \texorpdfstring{$L^2$}{L2}-bounded \texorpdfstring{$\ZT$}{Z/2}-harmonic 1-forms}
\begin{definition}
	Let $(M,g)$ be a closed Riemannian manifold, and let $\Sigma\subset M$ be a co-oriented
	embedded submanifold of codimension two. Let $\mathcal I$ be a flat real line bundle over
	$M\setminus\Sigma$ with monodromy $-1$ around small loops linking $\Sigma$.
	An \emph{$L^2$-bounded $\ZT$-harmonic $1$-form} is a triple $(\Sigma,\mathcal I,\omega)$, where
	$\omega\in \Omega^1(M\setminus\Sigma;\:\mathcal I)$ satisfies:
	\begin{enumerate}
		\item $\d \omega=0$ and $\d^\ast\!\omega=0$ (Here $\d$ is computed using the flat connection on
		$\mathcal I$ and $d^\ast$ is the formal adjoint with respect to $g$), and
		\item $\displaystyle \int_{M\setminus\Sigma}|\omega|^2<\infty$.
	\end{enumerate}
\end{definition}

The notion of $L^2$-bounded $\ZT$-harmonic $1$-forms is weaker than that of $\ZT$-harmonic $1$-forms
since it does not require $|\omega|$ to extend continuously across $\Sigma$. Nevertheless, there is a useful Hodge theory in this setting.

Let $\pi\colon \widehat M\to M\setminus\Sigma$ be the unique double cover that trivializes $\pi^\ast\mathcal I$,
and let $\hat g:=\pi^\ast g$ be the induced cone metric on $\widehat M$. Denote by $H^1_-(\widehat M)$ the
anti-invariant part of the de Rham cohomology under the $\mathbb{Z}_2$ action of the deck transformation. For an $L^2$-bounded
$\ZT$-harmonic $1$-form $(\Sigma,\mathcal I,\omega)$, the pullback $\hat\omega:=\pi^\ast\omega$ is an
ordinary $1$-form on $\widehat M$, anti-invariant under the deck transformation, and satisfies
\[
\d\hat\omega=0,\qquad \d^\ast_{\hat g}\hat\omega=0.
\]
By the $L^2$ Hodge theorem for Lipschitz metrics, $\hat\omega$ determines a cohomology class
$[\hat\omega]\in H^1_-(\widehat M)$.

Conversely, for every $\sigma\in H^1_-(\widehat M)$ there exists a unique $L^2$-bounded $\ZT$-harmonic
$1$-form $(\Sigma,\mathcal I,\omega)$ such that $[\pi^\ast\omega]=\sigma$
\cite{Donaldson2021,Hunsicker2005}. A convenient way to see existence is as follows: Let $U$ be a tubular neighbourhood of $\Sigma$ in $\widehat M$. Using the $\mathbb{Z}_2$-antisymmetry one has $H^1_-(U)=0$, and
so $\sigma$ admits a representative $\omega_{\mathrm{cpt}}\in\Omega^1(\widehat M)$ that is compactly
supported away from $\Sigma$. To obtain the harmonic representative, solve
\begin{equation}
	\label{eq:background:defining-equation-z2-on-double-cover}
	\Delta_{\hat g} u = \d^\ast \omega_{\mathrm{cpt}}
\end{equation}
for an anti-invariant function $u\in C^\infty(\widehat M)$. Then
\[
\widehat\omega := \omega_{\mathrm{cpt}} - \d u
\]
is $L^2$-harmonic on $(\widehat M,\hat g)$ representing $\sigma$. 
The push forward $\pi_*(\hat{\omega})$ is an $L^2$-bounded $\ZT$-harmonic $1$-form on $M$.

\subsection{Local expansions of \texorpdfstring{$\ZT$}{Z/2}-harmonic 1-forms}
\label{sec:background:local-expansions}
A basic analytic input in this subject is that $\ZT$-harmonic objects admit precise asymptotic
expansions near the singular set; see
\cite{BeraWalpuski2025Dirac, Donaldson2021,MazzeoHaydysTakahashi,Mazzeo1991}.
We briefly recall the form of these expansions and the associated leading coefficients.

We first set up some local coordinates.
Let $t=(t_i)$ be local coordinates on an open set in $\Sigma$, and let $U$ be the corresponding
tubular neighbourhood in $M$. On $U\setminus\Sigma$ we use polar normal coordinates
$(r,\phi,t)$, where $r$ is the distance to $\Sigma$ and $\phi$ is the angular variable in the
oriented normal plane. Setting $z=re^{i\phi}$ identifies the oriented normal bundle $N$ with a
complex line bundle. Using $\mathcal{I}$, we can also consider the bundle $N^{1/2}$ and equip it with the local coordinate $z^{1/2}=r^{1/2}e^{i\phi/2}$.

In these coordinates we can write down the asymptotic expansion near $\Sigma$.
Explicitly, if $u\in \Gamma(\mathcal I)$ is harmonic in a neighbourhood of $\Sigma$, then $u$ admits a
polyhomogeneous expansion of the form
\[
u \sim \sum_{p=0}^\infty \sum_{q=0}^p u_{pq}(t,\phi)\, r^{p+\frac12}\,(\log r)^q.
\]
Concretely, if
\(
u_k := \sum_{p=0}^{k} \sum_{q=0}^p u_{pq}(t,\phi)\, r^{p+\frac12}\,(\log r)^q,
\)
then $u-u_k=o(r^{k+1})$ as $r\to 0$, i.e.
\[
\lim_{r\to 0}\, \bigl\|r^{-k-1}(u-u_k)\bigr\|_{C^0(\Sigma\times\{r\}\times S^1)}=0.
\]

Donaldson \cite{Donaldson2021} also gave a description of the leading order terms of an $L^2$-bounded
$\ZT$-harmonic $1$-form $\omega$. Namely, on $U$ there exists a complex-valued function $A(t)$ such that
\[
\omega = \d\Re\bigl(A(t)\, z^{\frac12}\bigr) + o(1).
\]
Although this expression uses local coordinates, the coefficient $A$ is intrinsically a section of
$N^{-1/2}$. Moreover, $A\equiv 0$ if and only if $\omega$ is a genuine $\ZT$-harmonic $1$-form (i.e.\ when $|\omega|$ extends
continuously across $\Sigma$). In that case there exists a global section
$B\in \Gamma(N^{-3/2})$ such that
\[
\omega = \d\Re\bigl(B(t)\, z^{\frac32}\bigr) + o(r).
\]
A $\ZT$-harmonic $1$-form is nondegenerate in the classical sense if $B$ is nowhere vanishing.

\begin{remark}
Even when $A\not\equiv 0$, the next coefficient $B$ can still be defined \cite{Donaldson2021}.
In general there is a section $\mu\in\Gamma(N^{-1})$, determined by the mean curvature of $\Sigma$,
such that the expansion takes the form
\begin{equation}
	\label{eq:the-actual-expansion}
	\omega
	= d\,\Re\!\bigl(A(t)\,z^{\frac12}+B(t)\,z^{\frac32}\bigr)
	-\frac12\, d\,\Re\!\bigl(A(t)\,z^{\frac12}\bigr)\,\Re\!\bigl(\mu(t)\,z\bigr)
	+ o(r).
\end{equation}
We refer to $A$ and $B$ as the \emph{$A$--coefficient} and \emph{$B$--coefficient} of $\omega$.	

The coefficients $A$ and $B$ depend on the cohomology class of $\omega$ and on the metric $g$.
When it is helpful to emphasize this, we write $A(g,\cdot)$ and $B(g,\cdot)$, or
$A(g,[\omega],\cdot)$ and $B(g,[\omega],\cdot)$ respectively.
\end{remark}


Finally, to compare these coefficients under perturbations of $(g,\Sigma)$, one needs a
canonical identification of the normal data. Donaldson showed that, if $g$ is sufficiently close
to $g_0$ and $\Sigma$ is sufficiently close to $\Sigma_0$, then there exists a diffeomorphism
identifying the singular sets and their normal structures \cite[Proposition~4.2]{Donaldson2021}.
This allows one to view the expansions on a fixed normal bundle (and with respect to a fixed
reference metric) when studying variations.

\subsection{Decay behaviour and topological constraints.}
From the asymptotic expansion, one can read off the decay behaviour of $\omega$ near its singular set
$\Sigma$. Namely, if $\omega$ is a genuine $\ZT$-harmonic $1$-form then $|\omega|^{-1}=\O(r^{-1/2})$
near $\Sigma$, i.e.\ there exists a constant $C>0$ such that $|\omega|>C\,r^{1/2}$ near $\Sigma$.
Therefore, in this case $\Sigma$ is an isolated zero of $\omega$.

The same behaviour holds for $\vec{k}$-nondegenerate $\ZT$-harmonic $1$-forms. In this case there
exists an integer $l\in\Z$ such that $|\omega|^{-1}=\O(r^{\,l})$ near $\Sigma$. Moreover, this decay
behaviour is the only feature of $\vec{k}$-nondegenerate $\ZT$-harmonic $1$-forms that we will use in
this article.

However, such a decay behaviour does not occur for all $L^2$-bounded $\ZT$-harmonic $1$-forms. For
example:

\begin{example}
	On $\C_{z=re^{i\theta}}\times \mbR_t$, consider $\omega := \d\Re\bigl(t\,z^{\frac32}\bigr).$
	
	This is a $\ZT$-harmonic $1$-form with singular set $\Sigma=\{(0,t)\in \C\times \mbR\}$. However,
	the zero set of $\omega$ is not isolated; more precisely,
	\[
	|\omega|^{-1}(0)=\{(0,t)\}\cup \{(r=0,\theta=\tfrac{\pi}{3}+\tfrac{2m\pi}{3})\},
	\]
	for $m\in \Z$. In particular, $\omega$ is not $\vec{k}$-nondegenerate.
\end{example}

$\vec{k}$-nondegeneracy imposes not only an analytic condition, but also a mild and useful
restriction on the topology of the normal bundle. Namely, if $\omega$ is $\vec{k}$-nondegenerate
along $\Sigma$, then $N^{-(k+\frac12)}$ admits a nowhere vanishing section and hence is a trivial
complex line bundle. Equivalently, $(2k+1)c_1(N)=0\in H^2(\Sigma;\Z).$
Moreover, if a $\vec{k}$-nondegenerate $\ZT$-harmonic $1$-form can be deformed to a
$1$-nondegenerate one, then we additionally require $3c_1(N)=0$. Consequently, $\gcd(3,2k+1)\,c_1(N)=0,$
and note that $\gcd(3,2k+1)\in\{1,3\}$. In particular, if $\dim M=3,4$, then for a
$\vec{k}$-nondegenerate solution, $N$ must be a trivial bundle.

\subsection{Examples of \texorpdfstring{$\vec{k}$}{k}-nondegenerate \texorpdfstring{$\ZT$}{Z/2}-harmonic 1-forms.}
Finally, we present some known examples of $\vec{k}$-nondegenerate $\ZT$-harmonic $1$-forms.

\begin{example}[On a Riemann surface]
	Let $M$ be a Riemann surface with a complex structure induced by $g$ and let $q\in H^0(M,K_M^2)$ be a holomorphic quadratic differential. 
	Let $\Sigma$ be the set of zeros of $q$ of odd vanishing order.
	The real part of a local holomorphic square root $\sqrt{q}$, which is well-defined up to sign, defines a $\ZT$-harmonic $1$-form $\omega$ on $M\setminus \Sigma$.
	
	For any zero $p \in M$ of $q$ of odd vanishing order, and any choice of local holomorphic coordinate $z$ centred at $p$, one can write
	\[
	q = a_p\, z^{2k+1}\, dz\otimes dz + \text{higher order terms},
	\]
	where $a_p\neq 0$. This implies $\omega$ is $k$-nondegenerate along $p$.
\end{example}

\begin{example}[Ellipsoid in $\R^3$]
	There exists a $\vec{k}$-nondegenerate $\ZT$-harmonic $1$-form on $\R^3$ whose singular set $\Sigma$
	is an ellipsoid \cite{ChenHeYan26Calabi}.
\end{example}

\begin{example}[On a Seifert-fibered $3$-manifold]
	Let $\pi:Y\to \Sigma$ be a Seifert-fibered $3$-manifold, where $\Sigma$ is a $2$-dimensional orbifold
	equipped with an orbifold metric $g_{\Sigma}$. Let $i\eta$ be the connection $1$-form of a
	constant-curvature $U(1)$ connection on $Y$, and consider the Riemannian metric
	\[
	g_Y=\eta^2+\pi^*g_{\Sigma}
	\]
	on $Y$. Given an orbifold quadratic differential $q$ on $\Sigma$, the construction in
	\cite[Section~4.3]{HeParker2024} produces a $\ZT$-harmonic $1$-form on $Y$ associated to $\pi^*q$,
	which is $\vec{k}$-nondegenerate along the singular set.
\end{example}

\section{Metric perturbations of \texorpdfstring{$\ZT$}{Z/2} harmonic 1-forms}
\label{sec_metric_perturbation}
In this section we study how the $A$-- and $B$--coefficients of an $L^2$-bounded $\ZT$-harmonic
$1$-form vary under perturbations of the ambient metric. We show that suitable metric
perturbations allow one to prescribe the Taylor jet of the $B$--coefficient along $\Sigma$,
while forcing the $A$--coefficient to vanish to arbitrarily high order.

Before stating this precisely, we recall the notion of \emph{Schwartz functions}. In what follows,
we will repeatedly encounter remainder terms that are negligible to all orders as $r\to 0$ near
$\Sigma$.

\begin{definition}
	Let \(U\) be a neighbourhood of \(\Sigma\), and let \(r\) denote the distance function to \(\Sigma\).
	A section \(f\in \Gamma(\mathcal I|_{U\setminus\Sigma})\) is called \emph{Schwartz} on
	\(U\setminus\Sigma\) if \(f\) is smooth on \(U\setminus\Sigma\) and, for every pair of
	nonnegative integers \(p,q\), one has
	\[
	\sup_{U\setminus\Sigma}\, \bigl|\, r^{-p}\nabla^{q} f \,\bigr| < \infty,
	\]
	where \(\nabla\) denotes the covariant derivative induced by \(g\) together with the flat
	connection on \(\mathcal I\).
	We denote by \(\Msch\) the space of Schwartz sections of \(\mathcal I\) on \(U\setminus\Sigma\).
\end{definition}
In this section we prove the following metric-variation statement.

\begin{proposition}\label{prop:main-result-linear-theory}
	Let $(M,g_0)$ be a closed Riemannian manifold, and let $\Sigma\subset M$ be a smoothly embedded
	codimension-two submanifold with oriented normal bundle. Let $(\Sigma,\mathcal I,\omega_0)$ be a
	$\ZT$-harmonic $1$-form. Assume:
	\begin{enumerate}
		\item[(i)] $|\omega_0|>0$ on $U\setminus\Sigma$ for some tubular neighbourhood $U$ of $\Sigma$;
		\item[(ii)] for every $f\in\Msch$, one has $f/|\omega_0|\in \Msch$.
	\end{enumerate}
	Then, for any integer $m\ge 1$ and any prescribed smooth sections
	\[
	\tilde B_i \in \Gamma\!\bigl(N^{-(i+\frac12)}\bigr), \qquad 1\le i\le m,
	\]
	there exists a one-parameter family of smooth metrics $\{g_s\}_{s\in(0,1)}$ with $g_{s=0}=g_0$
	together with a family of $L^2$-bounded $\ZT$-harmonic $1$-forms $(\Sigma,\mathcal I,\omega_s)$ on
	$(M,g_s)$ representing the same cohomology class $[\omega_s]=[\omega_0]\in H^1_-(\hat M)$, such that
	the corresponding leading coefficients $A_s,B_s$ in \eqref{eq:the-actual-expansion} satisfy
	\begin{equation}
		\label{eq_variationABterms}
		\left.\pa_s^i A_s\right|_{s=0}=0,
		\qquad
		\left.\pa_s^i B_s\right|_{s=0}=\tilde B_i,
	\end{equation}
	for all $1\le i\le m$ and all $t\in\Sigma$. Moreover, for any fixed compact set $K\subset M\setminus
	\Sigma$, the family can be chosen so that $g_s|_K=g_0|_K$ for all $s$.
\end{proposition}

The assumptions in Proposition~\ref{prop:main-result-linear-theory} are only needed to justify
division by $|\omega_0|$ in the inductive step; in particular, they are automatic in the
$\vec{k}$-nondegenerate situation considered later. The remainder of this section is devoted to
the proof of Proposition~\ref{prop:main-result-linear-theory}. The proof proceeds by induction on
the jet order and combines linearized variation formulas, a local exactness statement for
$\ZT$-equivariant closed forms near $\Sigma$, and the construction of local models with prescribed
leading terms and Schwartz Laplacian.

\subsection{General Variation formulas}
As a start, we will compute the variation formula under metric perturbations.
For this we first need to set up some notation.
Fix a $\mathbb Z/2$-harmonic $1$-form $(\Sigma,\mathcal I,\omega_0)$ on a closed Riemannian manifold $(M,g_0)$ of dimension $n$. 
Let $g_s$ be a 1-parameter family of Riemann metric, that begins at $g_0$. 
Let $*_s$ be the Hodge star operator and $\Delta_{g_s}$ be the Laplace operator defined by $g_s$. Like before, let $\widehat{M}$ be the double cover over $M$ branched along $\Sigma$ that trivializes $\mathcal{I}$.

Fix $\omega_{cpt} \in \Omega^1(M \setminus \Sigma; \: \mathcal{I})$ such that $[\omega_0] = [\omega_{cpt}] \in H^1_-(\widehat{M})$ and $\omega_{cpt} = 0$ near $\Sigma$. 
Let $(\Sigma,\mathcal I,\omega_s)$ be the ($L^2$-bounded) $\ZT$-harmonic 1-form on $(M, g_s)$ such that $[\omega_s] = [\omega_{0}] \in H^1_-(\widehat{M})$. 
We can always write
$$
\omega_s = \d u_s + \omega_{cpt}
$$
for some $u_s \in \Gamma(\mathcal{I})$. Like in \eqref{eq:background:defining-equation-z2-on-double-cover},
$u_s$ is determined by the equation
\begin{equation}
	\label{eq_u_s}
	\Delta_{g_s} u_s=\d^{*_{s}}\omega_{cpt}.
\end{equation}
We also have an asymptotic expansion for $u_s$ like in \eqref{eq:the-actual-expansion}. Using that expansion, we denote $A_s(t) := A(g_s, [\omega_0], t)$ and $B_s(t) := B(g_s, [\omega_0], t)$.

To denote the derivatives, write $u_s^{(m)}:=\partial_s^m u_s$ and $g_s^{(m)}:=\partial_s^m g_s$.
Also, let $\omega_s^\sharp\in \Gamma(TM\otimes \mathcal{I})$ denote the $g_s$-dual section of the $1$-form $\omega_s$.
Finally, consider the 2-tensor
\begin{equation}\label{eq:Tms}
	\widehat T^{\langle m\rangle}_s
	:= g_s^{(m)} \;-\; \frac12\,\Tr_{g_s}(g_s^{(m)})\, g_s \;\in\; \Gamma(\operatorname{Sym}^2(T^*M)).
\end{equation}
\begin{proposition}\label{prop:linearized-equation}
	The first derivative $u_s^{(1)}$ satisfies
	\begin{equation}\label{eq_firstequation}
		\d^{*_s}\!\big(\!\d u_s^{(1)} - \widehat T^{\langle 1\rangle}_s(\omega_s^\sharp,\cdot)\big) = 0.
	\end{equation}
	More generally, for each $m\ge 1$,
	\begin{equation}
		\label{eq_higherorderequation}
		d^{*_s}\big(\!\d u_s^{(m)} - \widehat T^{\langle m\rangle}_s(\omega_s^\sharp,\cdot) \;+\;L^m_s\big) = 0, 
	\end{equation}
	where $L^m_s$ is a linear combination of lower-order terms depending on
	\[
	g_s,\; g_s^{(1)},\ldots,g_s^{(m-1)},\qquad
	\nabla u_s^{(1)},\ldots,\nabla u_s^{(m-1)},\qquad
	\omega_s^\sharp,
	\]
	and every term in $L^m_s$ contains at least one factor of $g_s^{(j)}$ for some $1\le j\le m-1$.
\end{proposition}

\begin{proof}
	Recall that for any $1$-form $v$ we have the variation formula
	\begin{equation}\label{eq_derivative_Hodge_star}
		\partial_s(*_s)v
		=
		*_s\!\left(\frac12\,\Tr_{g_s}(g_s^{(1)})\,v\;-\; g_s^{(1)}\!\big(v_s^\sharp,\cdot\big)\right),
	\end{equation}
	where $v_s^\sharp$ is the $g_s$-dual vector field of $v$.
	
	By \eqref{eq_u_s} we have $\d *_s \:\omega_s = \d*_s(du_s+\omega_{\mathrm{cpt}})=0$.
	Differentiating with respect to $s$ gives
	\[
	\d*_s \d u_s^{(1)}
	=
	- \d(\partial_s*_s)\:\omega_s
	=
	\d*_s\!\left(g_s^{(1)}(\omega^\sharp,\cdot)-\frac12\,\Tr_{g_s}(g_s^{(1)})\,\omega_s\right),
	\]
	where the last equality follows from \eqref{eq_derivative_Hodge_star}, and this implies \eqref{eq_firstequation}.

	For higher derivatives, differentiate \eqref{eq_u_s} $m$ times.
	The term $\d^{*_s}\!\d u_s^{(m)}=\Delta_{g_s}u_s^{(m)}$ is the only contribution involving $u_s^{(m)}$.
	All remaining terms arise from differentiating $\d^{*_s}$ (equivalently, $*_s$), and hence involve derivatives of the metric.
	The unique term containing $g_s^{(m)}$ appears linearly through $\widehat T^{\langle m\rangle}_s$,
	while every other term contains at least one factor among $g_s^{(1)},\ldots,g_s^{(m-1)}$ and at most derivatives
	$\nabla u_s^{(1)},\ldots,\nabla u_s^{(m-1)}$.
	This yields
	\begin{equation}
		\label{eq_higher_order_equations}
			\Delta_{g_s}u_s^{(m)}
		=
		\d^{*_s}\!\left(\widehat T^{\langle m\rangle}_s(\omega_s^\sharp,\cdot)+L^m_s\right),
	\end{equation}
	where $L^m_s$ consists of lower-order terms with the stated dependence.
\end{proof}

Notice that
$$
\Tr(\widehat{T}^{\langle m \rangle}_s) = \Tr(g^{(m)}_s) - \frac{1}{2} \Tr(g^{(m)}_s) \Tr(g_s) = \left(1 - \frac{n}{2}\right) \Tr(g^{(m)}_s).
$$
Hence, if $n \not = 2$, $g^{(m)}_s$ is fully determined by $\widehat{T}^{\langle m \rangle}_s$ by the relationship
\begin{equation}
	\label{eq_relationshipgandT}
	g^{(m)}_s = \widehat{T}^{\langle m \rangle}_s + \frac{1}{2-n}\Tr(\widehat{T}^{\langle m \rangle}_s) g_s.
\end{equation}
Therefore, if we find a satisfying solution of \eqref{eq_higherorderequation} in terms of $\widehat{T}^{\langle m \rangle}_0$, we automatically find a metric variation using \eqref{eq_relationshipgandT}. 

Instead of solving for $u^{(m)}_0$ in terms of $g^{(m)}_0$, we will find $g^{(m)}_0$ in terms of $u^{(m)}_0$. Using the De Rham cohomology on $\widehat{M}$, one can show there exists a family $\{\sigma^{ m }_0\}_{m \in \N} \in \Gamma(T^*M \otimes \mathcal{I})$, such that $\d^* \!\sigma^{ m }_0 = 0$ and
\begin{equation}
	\label{eq_firstequationAfterHodgeTheory}
	\widehat T^{\langle m\rangle}_0(\omega_0^\sharp,\cdot) =
	\d u_0^{(m)} + \sigma^{ m }_0 + L^m_0.
\end{equation}
We claim that wherever $\omega_0 \not = 0$, \eqref{eq_firstequationAfterHodgeTheory} can always be solved for $g^{(m)}_0$ using the following canonical choice:
\begin{lemma}\label{lem:canonical-choice-hat-T}
	Let $\alpha\in \Gamma(T^*M \otimes \MI)$ and define the symmetric $2$-tensor
	$T\in \Gamma(\Sym^2 T^*(M\setminus\Sigma))$ by
	\[
	T := \frac{\omega_0\otimes \alpha + \alpha\otimes \omega_0}{|\omega_0|^2} .
	\]
	Let $\widehat T := T-\frac12 \Tr_g(T)\,g$. Then, wherever $\omega_0 \not = 0$,
	\[
	\widehat T(\omega_0^\sharp,\cdot)=\alpha .
	\]
\end{lemma}
\begin{proof}
	For any vector field $X$ we have
	\[
	T(\omega_0^\sharp,X)= 
	\frac{\omega_0(\omega_0^\sharp)\alpha(X)+\alpha(\omega_0^\sharp)\omega_0(X)}{|\omega_0|^{2}}
	=\alpha(X)+ \frac{\langle \alpha,\omega_0\rangle}{|\omega_0|^{2}}\,\omega_0(X).
	\]
	Since $\Tr_g(\omega_0\otimes\alpha)=\langle \omega_0,\alpha\rangle$ and $\Tr_g(\alpha\otimes\omega_0)=\langle \alpha,\omega_0\rangle$,
	we have $\Tr_g(T)=\frac{2\langle \alpha,\omega_0\rangle}{|\omega_0|^2}$, hence
	\[
	\widehat T(\omega_0^\sharp,X)
	=
	T(\omega_0^\sharp,X)-\frac12\Tr_g(T)\,g(\omega_0^\sharp,X)
	=
	\alpha(X).
	\]
\end{proof}
Our goal is to find a solution for $\widehat{T}_0^{\langle m \rangle}$ that smoothly extend over the whole of $M$. Away from $\Sigma$, we can use bump functions to make $\widehat{T}_0^{\langle m \rangle}$ well-defined. Near $\Sigma$ this is not trivial. For example, a pointwise bound of \eqref{eq_firstequationAfterHodgeTheory} for the first derivative yields,
$$
\left|\d u_0^{(1)} + \sigma^{1}_0\right| \le |\widehat{T}^{\langle1 \rangle}_0| \cdot |\omega_0|
$$
Because we want to control $B_s$ in our variation, $u^{(1)}_0 = \O(r^{3/2})$. Also, we want that $\widehat{T}^{\langle 1 \rangle}_0$ extends to a smooth section over a compact manifold, and thus $\widehat{T}^{\langle 1 \rangle}_0 = \O(1)$. Finally, we need an estimate on $|\omega_0|^{-1}$ near $\Sigma$. For a degenerate $\ZT$-harmonic 1-form there is no restriction on $|\omega_0|^{-1}$, but for $\vec{k}$-nondegenerate $\ZT$ harmonic 1-forms, we have
$
|\omega_0|^{-1} = \O\left(r^{\frac12-\lceil\vec{k}\rceil}\right).
$
Therefore, for the first derivative we have to balance $\sigma^1_0$ such that 
$$
\left|\O(r^{\frac12}) + \sigma^{1}_0 \right| = \O(1)\cdot \O\left(r^{\lceil\vec{k}\rceil-\frac12}\right).
$$
To determine the freedom we have in $\sigma^1_0$, we first need to consider the cohomology of $\Z_2$-equivariant forms.
\begin{lemma}
	\label{lem:Z2-equivariant-closed-forms-are-exact-near-Sigma}
	Let $U \subset \widehat{M}$ be a tubular neighbourhood of the singular set $\Sigma$. Any closed form $\eta$ on $U$ that is anti-invariant under the deck transformation is exact.
\end{lemma}
\begin{proof}
	Retracting $U$ to the unit circle bundle $\pi: S \to \Sigma$, we consider the Gysin exact sequence:
	\[
	\cdots \longrightarrow H^k(\Sigma) \xrightarrow{\pi^*} H^k(U) \xrightarrow{\pi_*} H^{k-1}(\Sigma) \longrightarrow \cdots,
	\]
	where $\pi_*$ denotes integration along the fibre. Let $\tau: U \to U$ be the non-trivial deck transformation. Since $\tau$ acts on the $S^1$-fibers by rotation by $\pi$, it preserves the orientation of the fibres. Consequently, the integration map satisfies $\pi_* \circ \tau^* = \pi_*$.
	
	Let $[\eta] \in H^k(U)$ be an anti-invariant class, satisfying $\tau^*[\eta] = -[\eta]$. Applying the integration map, we find:
	\[
	\pi_*([\eta]) = \pi_*(\tau^*[\eta]) = -\pi_*([\eta]).
	\]
	Thus, $\pi_*([\eta]) = 0$. By exactness, $[\eta] = \pi^*(\alpha)$ for some $\alpha \in H^k(\Sigma)$. Since the projection $\pi$ is $\tau$-invariant (i.e., $\pi \circ \tau = \pi$), the image of $\pi^*$ lies entirely in the invariant subspace:
	\[
	\tau^*[\eta] = \tau^*(\pi^*\alpha) = (\pi \circ \tau)^*\alpha = \pi^*\alpha = [\eta].
	\]
	Comparing this with the assumption $\tau^*[\eta] = -[\eta]$, we conclude that $[\eta] = 0$.
\end{proof}
This lemma shows that $\sigma^m_0$ is co-exact on a tubular neighbourhood of $\Sigma$. Therefore, on this tubular neighbourhood, there exists a family $\{\tau^m_0\}_{m \in \N} \in \Gamma(\Lambda^2 T^*M \otimes \mathcal{I})$, such that

$$
\widehat T^{\langle m\rangle}_0(\omega_0^\sharp,\cdot) = \d u_0^{(m)} + \d^* \!\tau^{ m }_0 + L^m_0.
$$

\subsection{Perturbation using locally harmonic function}
Before proving Proposition~\ref{prop:main-result-linear-theory}, we first consider a simpler situation, closely related to \cite{Parker2022concentrating}, in which the metric variation is constant near the singular set. More precisely, let \(U\) be a small
neighbourhood of \(\Sigma\) and assume that the metric deformation satisfies \(g_s=g_0\) on \(U\).
Assume in addition that \(|\omega_0|>0\) on \(U\setminus\Sigma\). Under these hypotheses, the
deformation problem reduces to a local analytic question whether there exist
\(\mathcal I\)-valued harmonic functions near \(\Sigma\) with prescribed leading asymptotics. Explicitly, 

\begin{proposition}
	\label{prop:metric-variation-leading-term}
	Pick \(\widetilde A\in \Gamma(N^{-1/2})\) and \(\widetilde B\in \Gamma(N^{-3/2})\).
	There exists a family \(g_s\) such that for \eqref{eq_u_s} the leading expansion satisfies
	\[
	u_0^{(1)} = \Re \left(\widetilde A \, z^{\frac{1}{2}} + \widetilde B \, z^{\frac{3}{2}}\right)
	- \frac{1}{2} \Re \left( \widetilde  A\, z^{\frac{1}{2}}\right) \Re \left( \mu \, z\right)
	+ o(r^2)
	\]
	if and only if there exist a neighbourhood \(U\) of \(\Sigma\) and a harmonic section \(f\in \Gamma(\mathcal I_U)\) with
	\[
	f = \Re \left(\widetilde A \, z^{\frac{1}{2}} + \widetilde B \, z^{\frac{3}{2}}\right)
	- \frac{1}{2} \Re \left( \widetilde  A\, z^{\frac{1}{2}}\right) \Re \left( \mu \, z\right)
	+ o(r^2).
	\]
\end{proposition}

\begin{proof}
	\noindent\emph{Only if:}
	Since \(g_s\) is supported away from \(\Sigma\), differentiating \eqref{eq_u_s} at \(s=0\) gives
	\(\Delta_{g_0}\dot u_0=0\) on a neighbourhood of \(\Sigma\).
	Thus, \(\dot u_0\) is locally harmonic near \(\Sigma\), and its leading expansion is of the required form.
	
	\smallskip
	\noindent\emph{If:}
	Conversely, assume we are given such a harmonic \(f\) on \(U\).
	Since \(\d\:(*\d f)=0\), the \((n-1)\)-form \(*\d f\) is closed.
	By Lemma~\ref{lem:Z2-equivariant-closed-forms-are-exact-near-Sigma}, there exists \(h\in \Omega^{2}(\mathcal I_U)\) such that
	\[
	\d f + \d^*h = 0.
	\]
	Let \(\chi\) be a cutoff function supported in \(U\) with \(\chi\equiv 1\) near \(\Sigma\).
	Let $g_s = g_0 + s \: T^{\langle 1 \rangle}_0$ with 
	\begin{equation}\label{eq_hatT_0}
		T^{\langle 1 \rangle}_0
		:= \frac{2\,\omega_0\cdot\bigl(\d(\chi f)+\d^*(\chi h)\bigr)}{|\omega_0|^2}.
	\end{equation}
	By construction, $T^{\langle 1 \rangle}_0$ is supported in an annular region (where \(\d\chi\neq 0\)) and so it extends smoothly over $M$.
	Moreover, Lemma~\ref{lem:canonical-choice-hat-T} applied on \eqref{eq_hatT_0} yields
	\[
	\d^*\widehat T_0^{\langle 1\rangle}(\omega_0^\sharp,\ldots)=\Delta_{g_0}(\chi f).
	\]
	Then Proposition~\ref{prop:linearized-equation} gives
	\[
	\Delta_{g_0}(\dot u_0)=\d^*\widehat T_0^{\langle 1\rangle}(\omega_0^\sharp,\ldots)
	=\Delta_{g_0}(\chi f).
	\]
	Hence, \(\Delta_{g_0}(\dot u_0-\chi f)=0\) on \(M\).
	By \cite[Section~2]{Donaldson2021}, any bounded \(\mathcal I\)-valued harmonic section on \(M\) must vanish, and therefore \(\dot u_0=\chi f\).
	Therefore, they have the same asymptotic expansion.
\end{proof}

One can also prescribe higher-order derivatives of the leading coefficients, provided one has enough locally harmonic sections. For this we use that the error $L^m_s$ is compactly supported away from $\sigma$ whenever $g_s = g_0$ near $\Sigma$.

\begin{proposition}
	\label{prop:higher-order-variation}
	Fix $m \in \N$ and pick \(\widetilde A^1, \ldots \widetilde A^m \in \Gamma(N^{-1/2})\) and \(\widetilde B^1, \ldots \widetilde B^m\in \Gamma(N^{-3/2})\). There exists a family \(g_s\) such that for \eqref{eq_u_s} the leading expansion satisfies
	\[
	u_0^{(i)} = \Re \left(\widetilde{A}^i \, z^{\frac{1}{2}} + \widetilde{B}^i \, z^{\frac{3}{2}}\right)
	- \frac{1}{2} \Re \left( \widetilde{A}^i\, z^{\frac{1}{2}}\right) \Re \left( \mu \, z\right)
	+ o(r^2)
	\]
	for all $1 \le i \le m$
	if and only if there exist a neighbourhood \(U\) of \(\Sigma\) and a family of harmonic sections \(f^i\in \Gamma(\mathcal I_U)\) with
	\[
	f^i = \Re \left(\widetilde A^i \, z^{\frac{1}{2}} + \widetilde B^i \, z^{\frac{3}{2}}\right)
	- \frac{1}{2} \Re \left( \widetilde  A^i\, z^{\frac{1}{2}}\right) \Re \left( \mu \, z\right)
	+ o(r^2).
	\]
\end{proposition}
\begin{proof}
	We construct tensors \(T_i\) so that $g_s = g_0+\sum_{i=1}^m \frac{s^i}{i!}\,T_i$, and we proceed by induction.
	
	The case \(i=1\) is given by Proposition~\ref{prop:metric-variation-leading-term}.
	Assume \(T_i\) has been constructed for all \(i\le m-1\).
	Then in \eqref{eq_higherorderequation} the term \(L_0^m\) is well-defined, supported away from $\Sigma$ and smooth.
	
	Define
	\begin{equation}
		\label{eq_explicit_formula_Tm}
		\begin{split}
				T^{\langle m\rangle}_0
			:= \frac{2\,\omega_0\cdot\bigl(\d(\chi f_m)+\d^*(\chi h_m)-L_0^m\bigr)}{|\omega_0|^2}.
		\end{split}
	\end{equation}
	By construction this is a smooth section of \(\Sym^2(T^*M)\).
	Moreover, Lemma~\ref{lem:canonical-choice-hat-T} implies
	\[
	\widehat T^{\langle m\rangle}_0(\omega_0^\sharp,\ldots)
	= \d(\chi f)+\d^*(\chi h)-L_0^m.
	\]
	Therefore,
	\[
	\Delta_{g_0}\left(u^{(m)}_0\right)
	= \d^*\!\left(\widehat T^{\langle m\rangle}_0(\omega_0^\sharp,\ldots)+L_0^m\right)
	= \Delta_{g_0}(\chi f).
	\]
	Since \(f\) is bounded, the same argument as above implies \(u^{(m)}_0=\chi f\), which gives the desired identities for \(\partial_s^m A_s\) and \(\partial_s^m B_s\).
\end{proof}

\subsection{Schwartz approximation and deformations}
Given \(B\in \Gamma(N^{-3/2})\), one should not expect \(B\) to arise as the leading coefficient of
an \emph{exact} locally harmonic section near \(\Sigma\) for an arbitrary metric \(g\); this is a
genuinely metric-dependent constraint. For instance, when \(g\) is real-analytic, one expects
(though we are not aware of a proof) that locally harmonic sections have real-analytic leading
coefficients along \(\Sigma\). In particular, a non-real-analytic section \(B\) should not occur as
the leading term of an exact harmonic expansion. For this reason, the natural goal is to solve the
harmonic equation \emph{up to infinite order} near \(\Sigma\): we prescribe the leading behaviour while
allowing an error that is Schwartz along \(\Sigma\).

In this subsection, working in a neighbourhood \(U\) of \(\Sigma\subset M\), we construct a
\emph{Schwartz approximation} with prescribed leading expansion which solves the harmonic equation
to infinite order near \(\Sigma\). The construction is standard in the edge calculus and relies on
a Borel-type extension; see \cite{MazzeoHaydysTakahashi,Mazzeo1991}. Moreover, this Schwartz approximation will serve as an input for the metric deformations constructed later.

\begin{proposition}
	\label{prop_existence_schwartz_approximation}
	Given \(\widetilde B\in \Gamma(N^{-3/2})\), there exists a polyhomogeneous section
	\(u\in \Gamma(\mathcal I|_{U\setminus\Sigma})\) such that \(u= \widetilde B\, z^{3/2}+ o(r^2)\) and
	\(\Delta u\in \mathcal S\).
\end{proposition}
\begin{proof}
	The construction is standard in the edge calculus: we first build a formal polyhomogeneous approximation and then use the Borel lemma to obtain an actual polyhomogeneous section with the desired asymptotics.
	
	In local Fermi coordinates \((r,\theta,t)\) near \(\Sigma\), we write the Laplacian as \(\Delta=\Delta_0+E\), where
	\(\Delta_0=\partial_r^2+\partial_\theta^2+r^2\Delta_\Sigma\), and \(E\) is an error term which can be expressed schematically as a linear combination of
	\(\mathcal O(r^{-1})\{(r\partial)^2,\partial_\theta^2,r^2\partial_t^2\}\).

	Let \(L_0=(r\partial_r)^2+\partial_\theta^2\).
	For \(h=r^{k_1}e^{ik_2\theta}\), the equation \(L_0 f=h\) is solvable after allowing suitable \(\log r\) terms:
	if \(k_1\neq |k_2|\), one may take \(f=\frac{1}{k_1^2-k_2^2}h\), while if \(k_1=|k_2|\), one may take \(f=\frac{1}{2k_1}(\log r)\,h\).
	More generally, if \(h=P_{k_0}(\log r)\,r^{k_1}e^{ik_2\theta}\) with \(P_{k_0}\) a polynomial of degree \(k_0\), then one can find a polynomial \(Q\) of degree \(\le k_0\) (and of degree \(\le k_0+1\) in the resonant case \(k_1=|k_2|\)) such that
	\(f=Q(\log r)\,r^{k_1}e^{ik_2\theta}\) satisfies \(L_0 f=h\).
	Since solutions are not unique, we fix one choice and denote it by \(L_0^{-1}h\).
	
	We write \(r^2\Delta=L_0+R\), where \(R=r^2\Delta_\Sigma+r^2E\); note that \(R\) contributes higher order terms in \(r\) compared with \(L_0\).
	Set \(u_0:=B z^{3/2}\).
	Then \(L_0u_0=0\) and \(R(u_0)\sim r^{5/2}h+\mathcal O(r^{7/2})\) for some \(h\).
	Define \(u_1:=L_0^{-1}(r^{5/2}h)\); then \(\Delta(u_0+u_1)=\mathcal O(r^{7/2})\).
	Iterating this procedure, we obtain a formal polyhomogeneous series \(\hat u\sim u_0+u_1+\cdots\) with
	\(u_i\sim r^{3/2+i}\) (possibly with \(\log r\) terms), such that \(\Delta \hat u\) vanishes to infinite order as \(r\to 0\).
	
	Finally, we take a Borel sum to obtain the desired \(u\).
	Let \(\chi\) be a cutoff function supported in \(U\) with \(\chi\equiv 1\) near \(\Sigma\), and set \(\varepsilon_i:=e^{-i^2}\).
	Define $u:=\sum_{i=0}^{\infty}\chi\!\left(\frac{r}{\varepsilon_i}\right)u_i,$ then \(u\) is polyhomogeneous and satisfies \(\Delta u\in \mathcal S\), as claimed.
\end{proof}
In this rest of this subsection, we will show Proposition \ref{prop:main-result-linear-theory}  by repeating the argument in Proposition \ref{prop:metric-variation-leading-term} and Proposition \ref{prop:higher-order-variation} up to Schwartz error. As these propositions heavily depend on Lemma \ref{lem:Z2-equivariant-closed-forms-are-exact-near-Sigma}, we first need to generalize Lemma \ref{lem:Z2-equivariant-closed-forms-are-exact-near-Sigma} so that it allows a Schwartz error.

\begin{lemma}
	\label{lem:Z2-equivariant-Schwartz-forms-are-approximately-exact}
For $\eta\in \Omega^k(\MI|_U)$, if $\eta|_{\Sigma}=0$ and $d\eta\in \MS$, then for $\zeta := \int_{s \in (0,r)} \iota_{\del_r } \eta\in \Omega^{k-1}(\MI|_U)$, we have $\eta-d\zeta \in\MS$.	
\end{lemma}
\begin{proof}
	By the Cartan formula $\mathcal L_{\partial_r}\eta=d(\iota_{\partial_r}\eta)+\iota_{\partial_r}d\eta$, hence
	$$\eta(r)-\eta(0)=\int_0^r\mathcal L_{\partial_s}\eta(s)\,ds
	=d\zeta+\int_0^r \iota_{\partial_s}d\eta(s)\,ds.$$
	Since $\eta|_{\Sigma}=0$ we have $\eta(0)=0$, so we conclude
	$\eta-d\zeta=\int_0^r \iota_{\partial_s}d\eta(s)\,ds\in\MS$.
\end{proof}
\begin{proof}[Proof of Proposition \ref{prop:main-result-linear-theory}]
	The argument is identical to that of Proposition \ref{prop:metric-variation-leading-term} and
	Proposition \ref{prop:higher-order-variation}, with locally harmonic functions replaced by their
	Schwartz approximations. 
	
	For $m\in\{1,\ldots,k\}$, let $T^{\langle m\rangle}\in \Gamma(\Sym^2(T^*M))$ and set $g_s:=g_0+\sum_{m=1}^k \frac{s^m}{m!}\,T^{\langle m\rangle}.$ We determine $T^{\langle m\rangle}$ inductively. Given $B_m$, Proposition \ref{prop_existence_schwartz_approximation} provides a polyhomogeneous
	function $f_m$ on $U$ such that $\Delta f_m\in \MS$ and $f_m\sim \Re(\widetilde B_m \,z^{\frac32})+\cdots .$
	
	By Lemma \ref{lem:Z2-equivariant-Schwartz-forms-are-approximately-exact}, there exists a
	polyhomogeneous $2$-form $h_m\in \Omega^2(\MI|_U)$ such that $df_m+d^*h_m\in \MS$.
	Define $T^{\langle m\rangle}$ by \eqref{eq_explicit_formula_Tm}. By construction,
	$|\omega_0|^2\,T^{\langle m\rangle}$ is Schwartz; hence $T^{\langle m\rangle}$ is Schwartz as well and
	therefore extends smoothly across $\Sigma$ to a smooth tensor on $M$.
	
	Repeating the argument of Proposition \ref{prop:higher-order-variation}, we conclude $\left.\partial_s^m A_s\right|_{s=0}=0,\;
	\left.\partial_s^m B_s\right|_{s=0}=\widetilde{B}_m,$ as desired.
\end{proof}

\subsection{Deformations \texorpdfstring{$\vec{k}$}
{k}-nondegenerate \texorpdfstring{$\ZT$}{Z/2} harmonic 1-forms}
In this subsection we explain why the linear deformation result of
Proposition~\ref{prop:main-result-linear-theory} applies to any $\vec{k}$-nondegenerate
$\ZT$ harmonic $1$-form.  The only input is to verify the two technical assumptions in that proposition.

\begin{lemma}\label{lem:k-nondegenerate_satisfies_tech_assumptions}
	Let $(\Sigma,\mathcal I,\omega_0)$ be a $\vec{k}$-nondegenerate $\ZT$-harmonic $1$-form
	in the sense of Definition~\ref{def_k_nondegenerate}.
	Then the two technical assumptions in Proposition~\ref{prop:main-result-linear-theory} hold:
	\begin{enumerate}
		\item $|\omega_0|>0$ on $U\setminus \Sigma$ for some neighbourhood $U$ of $\Sigma$.
		\item For any function $f\in \Msch$, we have $f/|\omega_0|\in \Msch$.
	\end{enumerate}
\end{lemma}

\begin{proof}
	(1) By $\vec{k}$-nondegeneracy, in a tubular neighbourhood $U$ of $\Sigma$ we can write
	$\omega_0=du$ with
	\[
	u=\Re\!\big(B_k z^{k+\frac12}\big)+o\big(|z|^k\big),
	\]
	where $B_k$ is a nowhere vanishing section.
	A direct computation gives 
	$$ |\omega_0|=|du|\ge C\,\Big(\min_{\Sigma}|B_k|\Big)\,r^{k-\frac12}$$
	for $r$ sufficiently small and some constant $C>0$, hence $|\omega_0|>0$ on $U\setminus\Sigma$.
	
	\smallskip
	(2) Since $\omega_0$ is polyhomogeneous, all derivatives of $\omega_0$ have at most polynomial growth
	near $\Sigma$.
	If $f\in \Msch$ is rapidly decaying along $\Sigma$, then $f/|\omega_0|$ is also rapidly decaying by (1),
	and differentiating $f/|\omega_0|$ produces only sums of products of derivatives of $f$
	(still rapidly decaying) with derivatives of $|\omega_0|^{-1}$ (at most polynomially growing).
	Therefore, every derivative of $f/|\omega_0|$ is rapidly decaying, i.e.\ $f/|\omega_0|\in \Msch$.
\end{proof}

By Lemma~\ref{lem:k-nondegenerate_satisfies_tech_assumptions},
the assumptions of Proposition~\ref{prop:main-result-linear-theory} are satisfied for $\vec{k}$-nondegenerate $\ZT$-harmonic 1-forms. Therefore, applying Proposition~\ref{prop:main-result-linear-theory} on these 1-forms gives us the following corollary.

\begin{corollary}\label{cor:k-nondegenerate_apply_linear_theory}
	Let $(M,g_0)$ be a closed Riemannian manifold and let $\Sigma\subset M$ be a smoothly embedded
	codimension-two submanifold such that the normal bundle $N$ is oriented and $N^{-\frac32}$ is trivial.
	Let $(\Sigma,\mathcal I,\omega_0)$ be a $\vec{k}$-nondegenerate $\ZT$-harmonic $1$-form.
	Then, for any fixed integer $m\ge 1$ and any prescribed smooth sections
	\[
	\widetilde{B}_i \in \Gamma\!\big(N^{-\frac32}\big), \qquad 1\le i\le m,
	\]
	there exists a one-parameter family of smooth metrics $\{g_s\}$ with $g_{s}|_{s=0}=g_0$
	such that \eqref{eq_variationABterms} holds for all $1\le i\le m$.
	Moreover, for any fixed compact subset $K\subset M\setminus \Sigma$,
	the family can be chosen so that $g_s|_K=g_0|_K$ for all $s$.
\end{corollary}

\section{Deformations of k-nondegenerate \texorpdfstring{$\ZT$}{Z/2} harmonic 1-forms} 
In this section we prove the main existence result of the paper. The aim is to show that under certain conditions and after a suitable perturbation of the ambient metric, a degenerate $\ZT$-harmonic $1$-form can be deformed to a nondegenerate one.

Recall from Section~\ref{sec_metric_perturbation} that we constructed a one-parameter family of metrics $g_s$ which gives precise control of the Taylor expansions of the leading coefficients $A$ and $B$ along the singular set $\Sigma$. Following Donaldson~\cite{Donaldson2021}, the remaining step is to deform the singular set itself. We parameterize nearby singular sets by normal sections $v \in \Gamma(N)$ and consider the corresponding map $A(g, v)$ whose zeros encode solutions.

The main analytic input is a non-uniform Nash--Moser implicit function theorem. In our situation the relevant tame constants depend on the parameter $s$ (in particular, the inverse of the linearized operator and the quadratic error deteriorate as $s\to 0$), so we must work with a parameter-dependent Nash--Moser scheme and keep track of the $s$-dependence throughout.

We now state the main theorem.

\begin{theorem}\label{thm:main-theorem}
	Let $(\Sigma,\mathcal{I},\omega_0)$ be a degenerate $\ZT$-harmonic $1$-form on a closed Riemannian manifold $(M,g_0)$.
	Let $U$ be a tubular neighbourhood of $\Sigma$ and fix a compact set $K\subset M\setminus \Sigma$.
	Assume that
	\begin{enumerate}[label=(\roman*)]
		\item $|\omega_0|>0$ on $U\setminus \Sigma$,
		\item for every $f\in \Msch$, one has $f/|\omega_0|\in \Msch$, and
		\item there exist a section $\widetilde B\in \Gamma(N^{-3/2})$ and a constant $C>0$ such that for all $s\in(0,1)$ and all $t\in \Sigma$,
		\[
		\bigl|B(g_0)(t)+ s\,\widetilde B(t)\bigr|\ge Cs .
		\]
	\end{enumerate}
	There exist a one-parameter family of smooth metrics $\{g_s\}_{s\in(0,1)}$ 
	a deformation $\{\Sigma_s\}_{s\in(0,1)}$ of $\Sigma$, and a family of nondegenerate $\ZT$-harmonic $1$-forms
	\[
	(\Sigma_s,\mathcal{I},\omega_s)\quad\text{on}\quad(M,g_s),
	\]
	such that $g_s|_{s=0}=g_0$, $g_s|_K=g_0|_K$, and
	$\omega_s$ represents the same cohomology class as $\omega_0$ in $H^1_-(\widehat M)$.
\end{theorem}

\subsection{Donaldson's Deformation Theory}

We briefly recall the analytic structure behind Donaldson's deformation argument in~\cite{Donaldson2021}. Let $M$ be a compact $n$-manifold. Let $\mathcal{M}$ denote the space of smooth Riemannian metrics on $M$, and let $\mathcal{S}$ denote the space of smooth, codimension-2, co-oriented submanifolds of $M$. Both $\mathcal{M}$ and $\mathcal{S}$ are tame Fréchet manifolds.

Fix a reference submanifold $\Sigma \in \mathcal{S}$ and let $N$ denote its normal bundle. As discussed in \cite[Section 4]{Donaldson2021}, a neighbourhood of $\Sigma$ in $\mathcal{S}$ can be identified with a neighbourhood of the zero section in the space of sections $\Gamma(N)$. Specifically, using the normal exponential map, any section $v \in \Gamma(N)$ that is sufficiently small determines a nearby submanifold $\Sigma_v$. This identification allows us to parameterize local deformations of $\Sigma$ directly by sections of its normal bundle.

Let $(\Sigma_0, \mathcal{I}_0, \omega_0)$ be an $L^2$-bounded $\ZT$-harmonic 1-form over $(M, g_0)$. Let $U_1 \subset \mathcal{M}$ be a neighbourhood of $g_0$ and let $U_2 \subset \Gamma(N)$ be a neighbourhood of the zero section.
For any pair $(g, v) \in U_1 \times U_2$, let $\Sigma_v$ be the submanifold determined by $v$. There exists a unique $L^2$-harmonic 1-form $\omega$ on $(M, g)$ representing the same cohomology as $\omega_0$ in $H^1_-(\widehat{M})$. As explained in Section \ref{sec:background:local-expansions}, $\omega$ can be written as
\begin{equation}
	\label{eq:expansion}
	\omega
	= d\,\Re\bigl(A(t)\,z^{\frac12}+B(t)\,z^{\frac32}\bigr)
	-\frac12\, \d\,\Re\bigl(A(t)\,z^{\frac12}\bigr)\,\Re\bigl(\mu(t)\,z\bigr)
	+ o(r)
\end{equation}
using some suitable normal complex coordinate $z$.

Now, we fixed a cohomology class. Note that the coefficients $A$ and $B$ are originally defined over the perturbed submanifold $\Sigma_v$. To work within a fixed function space, we use the diffeomorphism induced by the normal exponential map to identify $\Sigma_v$ with $\Sigma_0$. This allows us to view $A$ and $B$ as sections over $\Sigma_0$; for technical details on this identification and the corresponding bundle maps, we refer to \cite[Section 4]{Donaldson2021}. Thus, we define the obstruction map
\[
A: U_1 \times U_2 \to \Gamma(\Sigma_0, N^{-1/2}), \quad (g, v) \mapsto A(g, v),
\]
and the secondary map
\[
B: U_1 \times U_2 \to \Gamma(\Sigma_0, N^{-3/2}), \quad (g, v) \mapsto B(g, v).
\]

Moreover, the maps $A$ and $B$ are smooth tame maps in the sense of Hamilton \cite{Hamilton1982}. This means that they satisfy linear growth estimates with respect to the graded norms of the Fréchet spaces. Specifically, there exists an integer $d \ge 0$ such that for every $i \ge 0$, we have a tame estimate of the form:
\begin{equation} \label{eq:tame_estimates}
	\| A(g, v) \|_{C^{i,\alpha}} \leq C_i \left( 1 + \| g \|_{C^{i+d,\alpha}} + \| v \|_{C^{i+d,\alpha}} \right).
\end{equation}
An analogous estimate holds for the secondary map $B$.

Donaldson computes the linearization of $A$ with respect to the normal variation $w \in \Gamma(N)$. He shows that the partial derivative takes the form:
\begin{equation} \label{eq:linearization}
	D_v A(g, v) \cdot w = \frac{3}{2} B w + \mathcal{R}(g, v, A, w),
\end{equation}
where $\mathcal{R}$ is a remainder term that depends linearly on $A$ and $w$.

Equation \eqref{eq:linearization} highlights the critical role of the secondary coefficient. If $A(g, v) = 0$, the remainder term $\mathcal{R}$ vanishes, and the linearization reduces to the multiplication map $w \mapsto \frac{3}{2} B w$. If $B$ is nowhere vanishing, this map is an algebraic isomorphism from $\Gamma(N)$ to $\Gamma(N^{-1/2})$.

However, the obstruction map $A$ is defined via the solution of an elliptic operator on the manifold, which introduces a `loss of derivatives' in the Fréchet topology. Standard Banach space methods fail here. Instead, Donaldson applies the Nash-Moser implicit function theorem. The key requirement is the existence of an `approximate inverse' with quadratic error.

\begin{theorem}[\cite{Donaldson2021}, Theorem 3]
	If $B$ is nowhere vanishing, then the map $D_v A$ admits a right inverse with quadratic error. Specifically, the map $\Psi(\eta) = \frac{2}{3} B^{-1} \eta $ serves as an approximate inverse. This means that for any $\eta \in \Gamma(N^{-1/2})$, we have
	\begin{equation} \label{eq:quadratic_error}
		D_v A(g, v) \left( \Psi(\eta) \right) = \eta + Q(g, v, A, \eta),
	\end{equation}
	where $Q$ is a smooth tame map that is bilinear in $A$ and $\eta$.
\end{theorem}

The local solvability in the non-degenerate case follows directly from this result. Since $B(g_0, 0)$ is assumed to be nowhere vanishing, it admits a uniform lower bound, which ensures that the multiplication operator $B^{-1}$ is bounded. Consequently, the approximate inverse $\Psi$ satisfies the necessary tame estimates to drive the convergence of the Newton iteration, yielding the following theorem:

\begin{theorem}[\cite{Donaldson2021}, Theorem 1]
	Suppose that $A(g_0, 0) = 0$ and that the non-degeneracy condition holds, i.e., $B(g_0, 0)$ is nowhere vanishing on $\Sigma$. Then there exist a neighbourhood $\mathcal{U} \subset \mathcal{M}$ of $g_0$ and a neighbourhood $\mathcal{V} \subset \Gamma(N)$ of the zero section such that for every metric $g \in \mathcal{U}$, there exists a unique normal section $v_g \in \mathcal{V}$ satisfying
	\[
	A(g, v_g) = 0.
	\]
	Moreover, the map $g \mapsto v_g$ is a smooth tame map from $\mathcal{U}$ to $\mathcal{V}$.
\end{theorem}

The central difficulty we address in this paper is the failure of the non-degeneracy condition. In our setting, the coefficient $B$ approaches zero as $s \to 0$, causing the uniform bound on $B^{-1}$ to fail. Specifically, under the family of metrics $g_s$ constructed in Theorem \ref{thm:main-theorem}, the coefficient $B_s$ behaves asymptotically like $s\,\widetilde{B}$. Consequently, the norm of the inverse scales as $\|B_s^{-1}\| \sim s^{-1}$.

This creates a critical competition: the error term $A_s$ (which we control via Section 3) must vanish at a sufficient rate to counteract the blow-up of the linearized operator as $s \to 0$. Standard perturbation arguments are insufficient here, as we must ensure that the solution $\Sigma_s$ remains within the domain of validity of the iteration scheme for every $s$. Therefore, our proof relies on a precise, quantitative version of the Nash-Moser theorem—carried out in the following subsection and Appendix A—where the dependence of all analytic estimates on the degeneration parameter $s$ is explicitly tracked.

\subsection{Quantitative Estimates for the Quantitative Nash--Moser}
In this subsection, we establish the quantitative estimates required for the Nash--Moser scheme in the degenerating regime $s \to 0$. To apply the scale-dependent implicit function theorem (Proposition A.2), we must carefully track the dependence of all relevant analytic constants on the parameter $s$. Specifically, we estimate the norm of the approximate inverse $\Psi_s$ (which diverges as $B_s \to 0$), the tame estimates for the error term $A_s$, and the quadratic error term.

Given a $\vec{k}$-nondegenerate $\ZT$ harmonic 1-form $(\omega,\MI,\Sigma)$. Recall from Proposition~\ref{prop:main-result-linear-theory} that for any fixed integer $m\ge 1$, we have constructed a one-parameter family of smooth metrics $\{g_s\}_{s\in(0,1)}$ extending $g_0$. In the notation of our obstruction map, the reference submanifold $\Sigma$ corresponds to the zero section $0 \in \Gamma(N)$. The construction in  ensures that the Taylor coefficients at the zero section satisfy:
\begin{equation}
	\label{eq_taylor_expansion_label_m}
	\left.\partial_s B(g_s, 0)\right|_{s=0}=\widetilde B,
	\qquad
	\left.\partial_s^i A(g_s, 0)\right|_{s=0}=0,\quad \left.\partial_s^{i+1}B(g_s, 0)\right|_{s=0}=0 \quad \text{for } 1\le i\le m.
\end{equation}
In particular, the coefficients admit the following asymptotic expansions:
\begin{align*}
	A(g_s, 0) = \mathcal{O}(s^{m+1}) \quad \text{and} \quad
	B(g_s, 0) = B_0 + s \,\widetilde B + \mathcal{O}(s^{m+1}).
\end{align*}
We now consider variations of the singular set parameterized by sections $v \in \Gamma(N)$. By restricting the general maps $A(g, v)$ and $B(g, v)$ to our specific family of metrics, we define the parameter-dependent operators:
\[
A_s(v) := A(g_s, v), \qquad B_s(v) := B(g_s, v).
\]
Since the bundle $N^{-3/2}$ is trivial, we shall regard $B_s(v)$ simply as a complex-valued function on $\Sigma$. With this notation, our goal is to find a small section $v$ such that $A_s(v)=0$. The term $A_s(0)$ represents the initial error at the reference submanifold, which is small at $\mathcal{O}(s^{m+1})$, but non-zero.

We now record the quantitative estimates needed for the Nash--Moser scheme in the degenerating
regime $s\to 0$. The key point is to make the $s$--dependence of all constants explicit.
We first estimate the decay rate of the linearized operator $B_s$.

\begin{lemma}\label{lem:B-is-invertible}
	Fix $m\in\N$ in \eqref{eq_taylor_expansion_label_m}. There exist $s_0>0$ and $C>0$ such that for all $0<s<s_0$,
	the function $B_s(0)\colon \Sigma\to\C$ is nowhere vanishing, hence invertible, and
	\[
	\|B_s(0)^{-1}\|_{C^0(\Sigma)} \le C\, s^{-1}.
	\]
\end{lemma}
\begin{proof}
	Using the integral remainder,
	\begin{equation}\label{eq:taylor-B}
		B_s(0)=B_0(0)+s\,\widetilde B + R_{m+1}(s),
		\text{ where }
		R_{m+1}(s):=\int_0^s \frac{(s-\sigma)^m}{m!}\,
		\left.\pa_s^{m+1}B_s(0)\right|_{s=\sigma}\,d\sigma .
	\end{equation}

	Since the map $s\mapsto g_s$ is smooth and $B(\,\cdot\,,\Sigma)$ is tame, the family $\{\pa_s^{m+1}B_s(0)\}_{s\in[0,1]}$ is bounded in $C^0(\Sigma)$. Hence, there exists
	$C_1>0$ such that
	\[
	\|R_{m+1}(s)\|_{C^0(\Sigma)}
	\le C_1\int_0^s (s-\sigma)^m\,d\sigma
	= \frac{C_1}{(m+1)!}\,s^{m+1}.
	\]
	
	On the other hand, assumption~(iii) in Theorem \ref{thm:main-theorem} provides a constant $c_0>0$ such that for all $t\in\Sigma$
	and all $s\in(0,1)$,
	\[
	|B_0(0)(t)+s\,\widetilde B(t)| \ge c_0\, s .
	\]
	Choose $s_0>0$ so that $\frac{C_1}{(m+1)!}\,s^{m}\le \frac{c_0}{2}$ whenever $0<s<s_0$.
	Then for all $t\in\Sigma$ and $0<s<s_0$,
	\[
	|B_s(0)(t)|
	\ge |B_0(0)(t)+s\,\widetilde B(t)| - |R_{m+1}(s)(t)|
	\ge c_0 s - \frac{C_1}{(m+1)!} s^{m+1}
	\ge \frac{c_0}{2}\, s .
	\]
	Thus, $B_s(0)$ is nowhere vanishing and $\|B_s(0)^{-1}\|_{C^0(\Sigma)} \le \frac{2}{c_0}\, s^{-1}.$
\end{proof}
Next, we prove uniform tame bounds for $A_s$ and its first two derivatives. These estimates are needed for Equation \eqref{eq:tame-estimates} in the appendix.
\begin{lemma}\label{lem:tame-estimates-As}
	Fix $\alpha\in\left(0,\frac12\right)$ and $m\in\N$ in \eqref{eq_taylor_expansion_label_m}.
	There exist $d\in\N$, $s_0\in(0,1)$, and constants
	$\delta,\,C_1,\,C_2,\,\{C_k\}_{k\ge d}>0$
	such that for all $s\in(0,s_0)$, all $k\ge d$, and all $u,v,w\in\Gamma(N)$ with
	$\|u\|_{C^{3d,\alpha}(\Sigma)}<\delta$, one has
	\begin{align*}
		\|A_s(u)\|_{C^{k,\alpha}(\Sigma)} &\le C_k\bigl(1+\|u\|_{C^{k+d,\alpha}(\Sigma)}\bigr), \qquad \forall k\ge d,\\
		\Bigl\|\frac{\del A_s(u)}{\del\Sigma}(v)\Bigr\|_{C^{2d,\alpha}(\Sigma)}
		&\le C_1\,\|v\|_{C^{3d,\alpha}(\Sigma)},\\
		\Bigl\|\frac{\del^2 A_s(u)}{\del\Sigma^2}(v,w)\Bigr\|_{C^{2d,\alpha}(\Sigma)}
		&\le C_2\,\|v\|_{C^{3d,\alpha}(\Sigma)}\,\|w\|_{C^{3d,\alpha}(\Sigma)}.
	\end{align*}
	Moreover, the loss of derivatives $d$ does not depend on $m$.
\end{lemma}
\begin{proof}
	Since $A$ is tame, there exist $d\in\N$ and constants $\{C_k'\}_{k\ge d}$ such that for $(g,u)\in U$ 
	\[
	\|A(g,u)\|_{C^{k,\alpha}(\Sigma)}
	\le C_k'\bigl(1+\|g-g_0\|_{C^{k+d,\alpha}(M)}+\|u\|_{C^{k+d,\alpha}(\Sigma)}\bigr),\qquad k\ge d.
	\]
	Moreover, after possibly shrinking $s_0$ we may assume that
	$\sup_{s\in(0,s_0)}\|g_s-g_0\|_{C^{k+d,\alpha}(M)}<\infty$ for each $k$.
	Absorbing this uniform bound into the constants yields the first estimate.
	
	Similarly, smooth tameness gives tame bounds for the first and second derivatives in the $u$--direction:
	there exist constants $\{\widetilde C_k\}_{k\ge d}$ and $\{\widehat C_k\}_{k\ge d}$ such that, for $(g_s,u)$ as above,
	\[
	\Bigl\|\frac{\del A_s(u)}{\del\Sigma}(v)\Bigr\|_{C^{2d,\alpha}}
	\le \widetilde C_{2d}\bigl(1+\|u\|_{C^{3d,\alpha}}\bigr)\bigl(1+\|v\|_{C^{3d,\alpha}}\bigr),
	\]
	\[
	\Bigl\|\frac{\del^2 A_s(u)}{\del\Sigma^2}(v,w)\Bigr\|_{C^{2d,\alpha}}
	\le \widehat C_{2d}\bigl(1+\|u\|_{C^{3d,\alpha}}\bigr)\bigl(1+\|v\|_{C^{3d,\alpha}}\bigr)\bigl(1+\|w\|_{C^{3d,\alpha}}\bigr).
	\]
	Using $\|u\|_{C^{3d,\alpha}}<\delta$ and Hamilton's standard rescaling argument
	\cite[Theorem II.2.1.5]{Hamilton1982} (applied to the linear maps $v\mapsto \frac{\del A_s(u)}{\del\Sigma}(v)$ and
	$(v,w)\mapsto \frac{\del^2 A_s(u)}{\del\Sigma^2}(v,w)$),
	we obtain homogeneous estimates of the form stated in the lemma, with constants $C_1,C_2$ depending on $\delta$ but not on $v,w$.
	
	Finally, the loss $d$ is determined by the tame structure of $(g,u)\mapsto A(g,u)$ and therefore is independent of $m$.
\end{proof}
Next, we introduce the inverse for the linearization by setting
\[
\Psi_{s}(u,v):=\frac{2}{3}\,B_{s}(u)^{-1}v,
\]
and define the associated quadratic error term by
\[
\mathcal{E}_{s}(u,w,v)
:=\mathcal{R}\bigl(g_{s}-g_0,\,u,\,w,\,\Psi_{s}(u,v)\bigr),
\]
where $\mathcal{R}$ is defined in \eqref{eq:linearization}.
With these choices, we have
\[
\frac{\delta A_{s}(u)}{\delta\Sigma}\bigl(\Psi_{s}(u,v)\bigr)
= v+\mathcal{E}_{s}\bigl(u, A_{s}(u), v\bigr),
\]
which is precisely the form of \eqref{eq:definition-inverse} in the appendix.
In the next steps we estimate the inverse $\Psi_s$ and the corresponding quadratic error
$\mathcal{E}_{s}$. This way we estimate the dependence of $\widetilde{C}_s$ and $\widetilde{C}_3$ on the parameter $s$ in \eqref{eq:estimate-inverse} and \eqref{eq:estimate-qadratic-error} in section \ref{sec:appendix:setup}.

\begin{lemma}
	\label{lem:tame-estimates-inverse-As}
	Fix $\alpha\in(0,\tfrac12)$ and $m\in\N$ in \eqref{eq_taylor_expansion_label_m}.
	There exist constants $d\in\N$, $s_0\in(0,1)$, $\delta>0$, and $\{\widetilde C_k\}_{k\ge d}>0$
	such that for all $s\in(0,s_0)$, $k\ge d$, $u\in\Gamma(N)$ and $v\in\Gamma(N^{-1/2})$
	with $\|u\|_{C^{3d,\alpha}(\Sigma)}< s\,\delta$, one has
	\[
	\|\Psi_s(u,v)\|_{C^{k,\alpha}(\Sigma)}
	\le \widetilde C_k\, s^{-k-2}\Bigl(
	\|v\|_{C^{k+d,\alpha}(\Sigma)}
	+\|u\|_{C^{k+d,\alpha}(\Sigma)}\,\|v\|_{C^{2d,\alpha}(\Sigma)}
	\Bigr).
	\]
	Moreover, the integer $d$ can be chosen uniformly in $m$.
\end{lemma}

\begin{remark}
In Lemma \ref{lem:tame-estimates-As}, we required that $\|u\|_{C^{3d,\alpha}(\Sigma)}<\delta$. Note that in Lemma \ref{lem:tame-estimates-inverse-As}, we additionally require that $\|u\|_{C^{3d,\alpha}(\Sigma)}<s\,\delta$. This is needed for the estimate of $B_s(u)^{-1}$.
\end{remark}

\begin{proof}[Proof of Lemma \ref{lem:tame-estimates-inverse-As}]
	This proof consists of several steps:
	
	\emph{Step 1: Get an explicit $C^0$ bound for $B_s(u)^{-1}$.}
	Like in Lemma~\ref{lem:tame-estimates-As}, one can find a uniform tame estimate for the $\Sigma$--derivative of $B_s$.
	Hence, there exists a constant $C_1 > 0$ such that for each $t\in\Sigma$ and $0< s \ll 1$,
	\[
	|B_s(u)(t)-B_s(0)(t)|
	\le \int_0^1 \bigl|D_\Sigma B_s(xu)[u](t)\bigr|\,dx
	\le C_1\|u\|_{C^{3d,\alpha}(\Sigma)}
	\le C_1 \,s\:\delta.
	\]
	By Lemma~\ref{lem:B-is-invertible} there is a $c_0>0$ such that
	$|B_s(0)(t)|\ge c_0 \,s$ for all $t\in\Sigma$ and all $s$ sufficiently small.
	Choosing $\delta\le \frac{c_0}{2 C_1}$ gives
	$|B_s(u)(t)|\ge \frac{c_0}{2}s$, and hence
	\[
	\|B_s(u)^{-1}\|_{C^0(\Sigma)}\le \frac{2}{c_0 \,s}.
	\]
	
	\emph{Step 2: Get an explicit $C^{k,\alpha}$ bounds for $B_s(u)^{-1}$.}
	Since the inversion map $f\mapsto f^{-1}$ is smooth on the open set
	$\{f:\ |f|\ge \frac{c_0}{2}s\}$, standard Hölder composition/product estimates
	(and the tame bounds for $B_s(u)$ from Lemma~\ref{lem:tame-estimates-As})
	yield constants $C_k>0$ such that, for all $k\ge d$ and $0 < s \ll 1$,
	\[
	\|B_s(u)^{-1}\|_{C^{k,\alpha}(\Sigma)}
	\le C_k\, s^{-k-2}\bigl(1+\|u\|_{C^{k+d,\alpha}(\Sigma)}\bigr).
	\]

	\emph{Step 3: Estimate $\Psi_s(u,v)$.}
	Using $\Psi_s(u,v)=\frac23\,B_s(u)^{-1}v$ and the standard Hölder product estimate, we get
	\[
	\|\Psi_s(u,v)\|_{C^{k,\alpha}}
	\le C_k\Bigl(\|B_s(u)^{-1}\|_{C^0}\|v\|_{C^{k,\alpha}}
	+\|B_s(u)^{-1}\|_{C^{k,\alpha}}\|v\|_{C^0}\Bigr).
	\]
	Combining this with the bounds from Steps 1--2 and $\|v\|_{C^0}\le \|v\|_{C^{2d,\alpha}}$
	gives the claimed inequality.
\end{proof}

\begin{lemma}
	\label{lem:tame-estimates-quadratic-error}
	Fix $\alpha\in(0,\tfrac12)$ and $m\in\N$ in \eqref{eq_taylor_expansion_label_m}.
	There exist constants $d\in\N$, $s_0\in(0,1)$ and $\widetilde C_3>0$ such that for all
	$s\in(0,s_0)$, $u\in\Gamma(N)$ and $v,w\in\Gamma(N^{-1/2})$ with
	$\|u\|_{C^{3d,\alpha}(\Sigma)}< s\,\delta$, one has
	\[
	\|\mathcal E_s(u,w,v)\|_{C^{2d,\alpha}(\Sigma)}
	\le \widetilde C_3\, s^{-3d}\,
	\|w\|_{C^{3d,\alpha}(\Sigma)}\,\|v\|_{C^{3d,\alpha}(\Sigma)}.
	\]
	Moreover, the integer $d$ does not depend on $m$.
\end{lemma}

\begin{proof}
	Recall
	\[
	\mathcal{E}_{s}(u,w,v):=\mathcal{R}\bigl(g_{s}-g_0,\,u,\,w,\,\Psi_{s}(u,v)\bigr),
	\]
	where $\mathcal{R}$ is Donaldson's quadratic error map, defined in \eqref{eq:linearization}.
	Since $\mathcal{R}$ is smooth tame, there exist $d'\in\N$ and constants $\{C_k\}_{k\ge d'}$
	such that for all $k\ge d'$,
	\begin{align}\label{eq:Q-tame}
		\|Q((g_s-g_0,u),w,\xi)\|_{C^{k,\alpha}(\Sigma)}
		\le C_k\,\Bigl(1+&\|g_s-g_0\|_{C^{k+d',\alpha}(M)}
		\\+&\|u\|_{C^{k+d',\alpha}(\Sigma)}\Bigr)\,
		\|w\|_{C^{k+d',\alpha}(\Sigma)}\,\|\xi\|_{C^{k+d',\alpha}(\Sigma)} . \notag
	\end{align}
	We apply \eqref{eq:Q-tame} with $k=2d$ and $\xi=\Psi_s(u,v)$:
	\begin{align*}
		\|\mathcal E_s(u,w,v)\|_{C^{2d,\alpha}}
		\le {}& C_{2d}\Bigl(1+\|g_s-g_0\|_{C^{2d+d',\alpha}}+\|u\|_{C^{2d+d',\alpha}}\Bigr)\,
		\|w\|_{C^{2d+d',\alpha}}\,
		\|\Psi_s(u,v)\|_{C^{2d+d',\alpha}} .
	\end{align*}
	
	Next we use Lemma~\ref{lem:tame-estimates-inverse-As}. Writing $d''$ for the loss
	appearing there, we have
	\begin{align*}
		\|\Psi_s(u,v)\|_{C^{2d+d',\alpha}}
		\le \widetilde C_{2d+d'}\, s^{-(2d+d')-2}\Bigl(
		\|v\|_{C^{2d+d'+d'',\alpha}}
		+\|u\|_{C^{2d+d'+d'',\alpha}}\|v\|_{C^{2d,\alpha}}
		\Bigr).
	\end{align*}
	Now choose $d$ large enough so that
	\[
	2d+d'+d''\le 3d
	\qquad\text{and}\qquad
	2d+d'+2\le 3d .
	\]
	Then $\|w\|_{C^{2d+d',\alpha}}\le \|w\|_{C^{3d,\alpha}}$,
	$\|v\|_{C^{2d+d'+d'',\alpha}}\le \|v\|_{C^{3d,\alpha}}$, and
	$s^{-(2d+d')-2}\le s^{-3d}$.
	Moreover, since $\|u\|_{C^{3d,\alpha}}<s\,\delta$ and $s<1$, we also have
	$\|u\|_{C^{3d,\alpha}}\le \delta$.
	Finally, $\|g_s-g_0\|_{C^{3d,\alpha}(M)}$ is uniformly bounded for $s\in(0,s_0)$.  
	Hence, the factor
	$1+\|g_s-g_0\|_{C^{2d+d',\alpha}}+\|u\|_{C^{2d+d',\alpha}}$
	can be absorbed into a constant.
	
	Putting these together yields
	\[
	\|\mathcal E_s(u,w,v)\|_{C^{2d,\alpha}}
	\le \widetilde C_3\, s^{-3d}\,
	\|w\|_{C^{3d,\alpha}}\,\|v\|_{C^{3d,\alpha}},
	\]
	for some $\widetilde C_3>0$.
	The choice of $d$ can be made independent of $m$ because it only needs to dominate
	the tame-loss constants $d',d''$ coming from $Q$ and $\Psi_s$, which are uniform in $m$.
\end{proof}

With the tame estimates from Lemmas~\ref{lem:tame-estimates-As},
\ref{lem:tame-estimates-inverse-As} and \ref{lem:tame-estimates-quadratic-error} in hand,
we can track how the auxiliary constants in
Section~\ref{sec:appendix:constants} depend on the scale parameter \(s\).
For simplicity, we assume that Lemmas~\ref{lem:tame-estimates-As}, \ref{lem:tame-estimates-inverse-As} and \ref{lem:tame-estimates-quadratic-error} hold with the same $d,\delta,$ and $s_0$. 
This can always be done by enlarging $d$, shrinking $\delta$ and $s_0$, and increasing the constants.
While most terms depending on \(u_0\), \(\delta\),
and \(\{g_s\}\) are uniformly bounded as \(s\to 0\), the loss coming from the inverse estimate contributes a factor
\(s^{-t-2}\) at order \(t\).
Substituting these bounds into the definitions in
Appendix~\ref{sec:appendix:constants}, we obtain
\[
\begin{aligned}
	\mathcal V_t
	&= \O(s^{-t-2})\Bigl(\O(1)+\O(1)\bigl(\O(1)+\O(s)\bigr)\Bigr)
	= \O(s^{-t-2}),\\
	V_0
	&= \O(s^{-d-2})\bigl(\O(1)+\O(s)\bigr)
	= \O(s^{-d-2}),\\
	V_1
	&= \O(s^{-T-2})\cdot \O(1)
	= \O(s^{-T-2}).
\end{aligned}
\]
Consequently, for the auxiliary bounds introduced in the Appendix~\ref{sec:appendix:constants} we have
\[
\begin{aligned}
	V
	&= \O(s^{-d-2})+\O(s^{-T-2})
	= \O(s^{-T-2}),\\
	U_t
	&= \O(1)+\O(s^{-t-2})
	= \O(s^{-t-2}),\\
	C_0
	&= \O(s^{-T-2})+\O(s^{-2T-4})
	+\O(s^{-3d})\cdot\bigl(\O(1)\bigr)^2
	= \O(s^{-2T-4}).
\end{aligned}
\]
Therefore, the scale parameter \(\theta_0\) may be chosen so that
\[
\begin{aligned}
	\theta_0
	&= \max\Bigl\{\O(s^{-2T-4}),\,\O(s^{-T-2}),\,\O(s^{-\tau+d-2}),\,\O(s^{-T-3})\Bigr\}\\
	&= \O\bigl(s^{-2(T+2)}\bigr),
\end{aligned}
\]
where \(T = 16d^2+41d+24\) and \(\tau = 8d^2+10d\).
Using Proposition~\ref{prop:explicit-nash-moser}, we can now complete the proof of
Theorem~\ref{thm:main-theorem}.

\begin{proof}[Proof of Theorem~\ref{thm:main-theorem}]
	By Proposition~\ref{prop:explicit-nash-moser}, there exist constants \(C>0\) and
	\(s_0\in(0,1)\), such that if \(s\in(0,s_0)\) and
	\[
	\|A_{s}(0)\|_{C^{2d,\alpha}(\Sigma)}< C\, s^{\,8(16d^2+41d+26)},
	\]
	there exists a smooth section \(u\in\Gamma(N)\) with \(A_{s}(u)=0\).
	By construction,
	\[
	\|A_{s}(0)\|_{C^{2d,\alpha}(\Sigma)}=\O(s^{m+1}).
	\]
	Moreover, Lemmas~\ref{lem:tame-estimates-As}, \ref{lem:tame-estimates-inverse-As} and
	\ref{lem:tame-estimates-quadratic-error} show that the loss parameter \(d\) can be chosen
	independently of \(m\). Hence, if
	\[
	m := 8(16d^2+41d+26),
	\]
	the inequality \(\|A_{s}(0)\|_{C^{2d,\alpha}}<\C\, s^{\,8(16d^2+41d+26)}\) holds for all sufficiently small \(s\), and so the hypotheses of Proposition~\ref{prop:explicit-nash-moser} are satisfied.
	
	Finally, Lemma~\ref{lem:tame-estimates-inverse-As} gives
	\(\|B^{-1}_{s}(u)\|_{C^0(\Sigma)}=\O(s^{-1})\), which implies that the resulting
	\(\ZT\)-harmonic \(1\)-forms are non-degenerate.
\end{proof}
With these estimates in hand, we can now prove the main result,
Theorem~\ref{thm:k-degenerate-theorem}.

\begin{proof}[Proof of Theorem~\ref{thm:k-degenerate-theorem}]
	Let \((\Sigma_0,\MI,\omega_0)\) be a \(k\)-nondegenerate \(\ZT\)-harmonic \(1\)-form on
	\((M,g)\). By Corollary~\ref{cor:k-nondegenerate_apply_linear_theory}, there exists a
	one-parameter family of smooth metrics \(\{g_s\}\) such that the associated
	\(L^2\)-harmonic \(1\)-forms \((\Sigma,\MI,\omega_s)\) satisfy assumptions~(i) and~(ii) of
	Theorem~\ref{thm:main-theorem}.  
	
	Since \(N^{-3/2}\) is a trivial complex line bundle, assumption~(iii) of
	Theorem~\ref{thm:main-theorem} is automatic. Therefore, Theorem~\ref{thm:main-theorem}
	applies and yields the desired conclusion.
\end{proof}

\appendix

\section{A quantitative Nash-Moser theory with quadratic errors}
\label{sec:appendix}
In this paper we require a Nash--Moser implicit function theorem with quadratic
error\footnote{in the sense of Hamilton~\cite{Hamilton1982}.}, adapted to a \emph{non-uniform} scale. Related quantitative estimates are implicit in the arguments of
Donaldson--Lehmann~\cite{DonaldsonLehmann2024}, building on earlier work of
Raymond~\cite{Raymond1989}; see also~\cite{baldi2017nash} for a discussion of similar
scale-sensitive issues. However, for our application it is essential to track the dependence
of every constant appearing in the tame estimates and in the quadratic error bounds. In
particular, we must control the size of the domain on which the Nash--Moser scheme converges
in a scale-dependent (hence non-uniform) manner.

The purpose of this appendix is to collect these bounds in a single place and to state them
explicitly in a form convenient for later use. To make the
dependence of constants transparent, we also include self-contained proofs following the
approach of~\cite{DonaldsonLehmann2024,Raymond1989}.

\subsection{The setup for the Nash--Moser theorem}
\label{sec:appendix:setup}
Let $\F$ and $\G$ be Fr\'echet spaces equipped with graded norms $|\cdot|_s$ and $\|\cdot\|_s$,
respectively.  Let $\phi\colon \F\to \G$ be a smooth map, and fix a base point
$u_0\in \F$.  Throughout we work in a neighbourhood of $u_0$ on which $\phi$ satisfies
tame estimates of order $d$.

\medskip
\noindent
\textit{Tame bounds for $\phi$:}
Assume that there exist an integer $d\in\mathbb{N}$ and constants
\[
\delta>0,\qquad C_1>0,\qquad C_2>0,\qquad \{C_s\}_{s\ge d}\subset (0,\infty),
\]
such that for all $u,v,w\in \F$ with $|u-u_0|_{3d}<\delta$ one has
\begin{equation}
	\label{eq:tame-estimates}
	\begin{cases}
		\forall s\ge d,\quad \|\phi(u)\|_s \le C_s\bigl(1+|u|_{s+d}\bigr),\\[2mm]
		\|\phi'(u)v\|_{2d}\le C_1\,|v|_{3d},\\[2mm]
		\|\phi''(u)(v,w)\|_{2d}\le C_2\,|v|_{3d}\,|w|_{3d}.
	\end{cases}
\end{equation}

\medskip
\noindent
\textit{An approximate right inverse and quadratic error:}
Assume that for each $u\in \F$ with $|u-u_0|_{3d}<\delta$ there exist a linear operator
$\psi(u)\colon \G\to \F$ and a bilinear operator
$\mathcal{E}(u)\colon \G\times \G\to \G$
such that for every $v\in \G$,
\begin{equation}
	\label{eq:definition-inverse}
	\phi'(u)\,\psi(u)v \;=\; v + \mathcal{E}(u)\bigl(\phi(u),v\bigr).
\end{equation}
Moreover, suppose there exist constants $\widetilde C_3>0$ and
$\{\widetilde C_s\}_{s\ge d}\subset (0,\infty)$ such that for all $s\ge d$,
\begin{align}
	\label{eq:estimate-inverse}
	|\psi(u)v|_s
	&\le \widetilde C_s\Bigl(\|v\|_{s+d} + |u|_{s+d}\,\|v\|_{2d}\Bigr),\\
	\label{eq:estimate-qadratic-error}
	\|\mathcal{E}(u)(v,w)\|_{2d}
	&\le \widetilde C_3\,\|v\|_{3d}\,\|w\|_{3d}.
\end{align}
Hamilton's notion of \emph{quadratic error} implies the form \eqref{eq:estimate-qadratic-error}
after a standard rescaling argument \cite[Theorem~II.2.1]{Hamilton1982}.

\medskip
\noindent
\textit{The smoothing operators:}
Finally, assume that there exists a family of smoothing operators
$(S_{\theta})_{\theta>1}\colon \F\to \F$ such that for all $v\in \F$, $\theta>1$,
and all $s,t\ge 0$,
\begin{align}
	|S_\theta v|_s &\le \widehat{C}_{s,t}\,\theta^{\,s-t}\,|v|_t,
	\qquad &&\text{if } s\ge t, \notag\\
	|v-S_\theta v|_s &\le \widehat{C}_{s,t}\,\theta^{\,s-t}\,|v|_t,
	\qquad &&\text{if } s\le t.
	\label{eq:estimates-smoothing-operators}
\end{align}

\begin{remark}
	Whenever possible we follow the notation of Raymond~\cite{Raymond1989}.
	However, Raymond uses the same symbols $C_s$ for the tame bounds on $\phi$,
	for the bounds on the approximate inverse, and for the smoothing estimates.
	Here we separate these families of constants by using different diacritics.
\end{remark}

\subsection{Constants}
\label{sec:appendix:constants}

Under the hypotheses of Section~\ref{sec:appendix:setup}, Raymond~\cite{Raymond1989}
(and subsequently Donaldson--Lehmann~\cite{DonaldsonLehmann2024}) introduce a collection of
auxiliary constants which govern the loss of derivatives and the size of the domain of
validity of the Nash--Moser scheme.  We record these constants here in a single place, with
notation consistent with \cite{Raymond1989} up to the diacritics introduced above.

\medskip
\noindent
\textit{Grading parameters:}
Define
\[
N := 4(2d+1),\qquad
T := 3d + 3 + (2d+3)(N+3),\qquad
\tau := 3d + d(N+3).
\]

\medskip
\noindent
\textit{Auxiliary bounds:}
Next, define for $t\ge d$,
\[
\mathcal{V}_t :=
\widetilde{C}_t\Bigl( C_{t+d} + C_{2d}\bigl(1+\delta+|u_0|_{3d}\bigr)\Bigr),
\qquad
V_0 := \widetilde{C}_d\bigl(1+\delta+|u_0|_{2d}\bigr),
\]
and set
\[
V_1 := \mathcal{V}_T\bigl(1+|u_0|_{T+2d}\bigr),\qquad
V := \widehat{C}_{3d+3,d}\,V_0 + \widehat{C}_{3d+3,T}\,V_1,
\qquad
U_t := 1+\widehat{C}_{t+2d,t}\,\mathcal{V}_t.
\]
Finally, let
\[
C_0 :=
C_1\,\widehat{C}_{3d,\,3d+3}\,V
+\widehat{C}_{3d,\,3d}^2\,V^2
+\widetilde{C}_3\Bigl(
\widehat{C}_{3d,\,2d}
+\widehat{C}_{3d,\,\tau}\,C_\tau\bigl(1+|u_0|_{\tau+d}\bigr)
\Bigr)^2.
\]

\medskip
\noindent
\textit{Scale sequence:}
The size of the admissible neighbourhood is encoded in a sequence $(\theta_k)_{k\ge 0}$.
Set
\[
\theta_0 := \max\Bigl\{\,2,\; U_T,\; U_{\tau-d},\; \widehat{C}_{3d,\,3d}\,\frac{V}{\delta},\; C_0\,\Bigr\},
\qquad
\theta_{k+1} := \theta_k^{5/4}.
\]
As explained in \cite{Raymond1989}, $(\theta_k)$ is increasing and satisfies the summability estimate
\[
\sum_{j\ge 0}\theta_j^{-3} < \theta_0^{-1}.
\]

\medskip
\noindent
With these constants fixed, the Nash--Moser theorem with quadratic error can be stated as follows.

\begin{proposition}[{\cite[Main theorem]{Raymond1989}; \cite[Theorem~A.11]{DonaldsonLehmann2024}}]
	\label{prop:explicit-nash-moser}
	Consider the setup of Section~\ref{sec:appendix:setup}, and let $\theta_0$ be as in
	Section~\ref{sec:appendix:constants}. If
	\[
	\|\phi(u_0)\|_{2d} \le \theta_0^{-4},
	\]
	then there exists $u\in \F$ such that $|u-u_0|_{3d}<\delta$ and $\phi(u)=0$.
\end{proposition}

\subsection{Iteration scheme}
Raymond~\cite{Raymond1989} constructs a solution of $\phi(u)=0$ via the iterative scheme
\begin{equation}\label{eq:raymond-iteration}
	v_k := -\psi(u_k)\phi(u_k),
	\qquad
	u_{k+1} := u_k + S_{\theta_k} v_k .
\end{equation}
In this subsection we show that $(u_k)$ converges in the $|\cdot|_{3d+d}$--norm and that its limit
is a zero of $\phi$.  The following estimate is one of the basic ingredients.

\begin{lemma}[{\cite[Lemma~1(iii)]{Raymond1989}}]
	\label{lem:property-3-Raymond}
	Fix $k\in \mathbb{N}$. Assume that for all $0\le j\le k$,
	\[
	|u_j-u_0|_{3d}<\delta
	\qquad\text{and}\qquad
	\|\phi(u_j)\|_{2d}\le \theta_0^{-4}.
	\]
	Then for every $t\ge d$,
	\[
	1+|u_{k+1}|_{t+2d}
	\le U_t\,\theta_k^{2d}\,\bigl(1+|u_k|_{t+2d}\bigr).
	\]
\end{lemma}

\begin{proof}
	We combine the tame bounds \eqref{eq:tame-estimates} with the inverse estimate
	\eqref{eq:estimate-inverse}.  For $t\ge d$,
	\begin{align*}
		|v_k|_t
		&\le \widetilde C_t\Bigl(\|\phi(u_k)\|_{t+d} + |u_k|_{t+d}\,\|\phi(u_k)\|_{2d}\Bigr) \\
		&\le \widetilde C_t\Bigl( C_{t+d}\bigl(1+|u_k|_{t+2d}\bigr)
		+ |u_k|_{t+d}\, C_{2d}\bigl(1+|u_k|_{3d}\bigr)\Bigr).
	\end{align*}
	Using $|u_k|_{t+d}\le 1+|u_k|_{t+2d}$ and the assumption $|u_k-u_0|_{3d}<\delta$ (hence
	$|u_k|_{3d}\le |u_0|_{3d}+\delta$), we obtain
	\begin{align*}
		|v_k|_t
		&\le \widetilde C_t\Bigl( C_{t+d}+ C_{2d}\bigl(1+\delta+|u_0|_{3d}\bigr)\Bigr)\,
		\bigl(1+|u_k|_{t+2d}\bigr) \\
		&=: \mathcal V_t\,\bigl(1+|u_k|_{t+2d}\bigr).
	\end{align*}
	Thus,
	\begin{equation}
		\label{eq:estimate-that-defines-Wt}
		|v_k|_t \le \mathcal V_t\bigl(1+|u_k|_{t+2d}\bigr).
	\end{equation}
	
	Next, from \eqref{eq:raymond-iteration} and the smoothing estimates
	\eqref{eq:estimates-smoothing-operators}, we have
	\[
	|u_{k+1}|_{t+2d}
	\le |u_k|_{t+2d} + |S_{\theta_k}v_k|_{t+2d}
	\le |u_k|_{t+2d} + \widehat C_{t+2d,t}\,\theta_k^{2d}\,|v_k|_t.
	\]
	Combining this with \eqref{eq:estimate-that-defines-Wt} yields
	\begin{align*}
		1+|u_{k+1}|_{t+2d}
		&\le 1+|u_k|_{t+2d}
		+ \widehat C_{t+2d,t}\,\theta_k^{2d}\,\mathcal V_t\,\bigl(1+|u_k|_{t+2d}\bigr) \\
		&\le \bigl(1+\widehat C_{t+2d,t}\,\mathcal V_t\bigr)\,\theta_k^{2d}\,\bigl(1+|u_k|_{t+2d}\bigr),
	\end{align*}
	where in the last step we used $\theta_k^{2d}\ge 1$.
	Recalling the definition $U_t := 1+\widehat C_{t+2d,t}\,\mathcal V_t$, the claim follows.
\end{proof}

\begin{lemma}[{\cite[Lemma~1(ii)]{Raymond1989}}]
	\label{lem:property-2-Raymond}
	Fix $k\in\mathbb{N}$. Assume that for all $0\le j\le k$,
	\[
	|u_j-u_0|_{3d}<\delta
	\qquad\text{and}\qquad
	\|\phi(u_j)\|_{2d}\le \theta_0^{-4}.
	\]
	Then
	\[
	|v_k|_{3d+3}\le V\,\theta_k^{-3}.
	\]
\end{lemma}
\begin{proof}
	Using \eqref{eq:estimate-inverse} with $s=d$ and $|u_k-u_0|_{3d}<\delta$ (so $|u_k|_{2d}\le |u_0|_{2d}+\delta$), we obtain
	\begin{equation}\label{eq:Vk_d_bound}
		|v_k|_d = |\psi(u_k)\phi(u_k)|_d
		\le \widetilde C_d(1+|u_k|_{2d})\|\phi(u_k)\|_{2d}
		\le V_0\,\|\phi(u_k)\|_{2d}.
	\end{equation}
	Since $\|\phi(u_k)\|_{2d}\le \theta_0^{-4}$ and $\theta_k\ge \theta_0$, it follows that
	\begin{equation}\label{eq:Vk_d_bound_theta}
		|v_k|_d \le V_0\,\theta_k^{-4}.
	\end{equation}
	
	Recall $N:=4(2d+1)$ and $T:=3d+3+(2d+3)(N+3)$. We claim that
	\begin{equation}\label{eq:uj_growth_T}
		1+|u_j|_{T+2d}\le (1+|u_0|_{T+2d})\,\theta_j^{N}\qquad\text{for all }0\le j\le k.
	\end{equation}
	Indeed, the case $j=0$ is immediate. If \eqref{eq:uj_growth_T} holds for some $j<k$, then Lemma~\ref{lem:property-3-Raymond} with $t=T$ gives
	\[
	1+|u_{j+1}|_{T+2d}\le U_T\,\theta_j^{2d}\,(1+|u_j|_{T+2d})
	\le U_T\,\theta_j^{N+2d}\,(1+|u_0|_{T+2d}),
	\]
	and using $\theta_{j+1}=\theta_j^{5/4}$ and $N=4(2d+1)$ we have $\theta_j^{N+2d}=\theta_j^{-1}\theta_{j+1}^{N}\le \theta_0^{-1}\theta_{j+1}^{N}$; assuming $\theta_0\ge U_T$ closes the induction.
	
	Combining \eqref{eq:estimate-that-defines-Wt} (with $t=T$) and \eqref{eq:uj_growth_T} yields
	\begin{equation}\label{eq:Vk_T_bound}
		|v_k|_T \le \mathcal V_T(1+|u_k|_{T+2d})
		\le \mathcal V_T(1+|u_0|_{T+2d})\,\theta_k^{N}
		= V_1\,\theta_k^{N}.
	\end{equation}
	
	Set $\eta_k:=\theta_k^{1/(2d+3)}$ and write $v_k=S_{\eta_k}v_k+(1-S_{\eta_k})v_k$. Then \eqref{eq:estimates-smoothing-operators} gives
	\begin{align*}
		|v_k|_{3d+3}
		&\le \widehat C_{3d+3,d}\,\eta_k^{2d+3}|v_k|_d+\widehat C_{3d+3,T}\,\eta_k^{3d+3-T}|v_k|_T \\
		&= \widehat C_{3d+3,d}\,\theta_k|v_k|_d+\widehat C_{3d+3,T}\,\theta_k^{-(N+3)}|v_k|_T \\
		&\le \bigl(\widehat C_{3d+3,d}V_0+\widehat C_{3d+3,T}V_1\bigr)\theta_k^{-3}
		= V\,\theta_k^{-3},
	\end{align*}
	where we used \eqref{eq:Vk_d_bound_theta} and \eqref{eq:Vk_T_bound} in the last line.
\end{proof}

Heuristically, $v_k$ is the correction at the $k$-th step; in particular,
\[
u_{k+1}-u_k = S_{\theta_k}v_k
\]
is a smoothed version of $v_k$. Lemma~\ref{lem:property-2-Raymond} shows that, as long as the
iterates $u_k$ stay in the tame neighbourhood and $\phi(u_k)$ decays at the required rate, the
corrections $v_k$ decay accordingly. We now verify that the assumptions of
Lemmas~\ref{lem:property-3-Raymond} and~\ref{lem:property-2-Raymond} hold for all $k$.

\begin{lemma}[{\cite[Lemma~1(i)]{Raymond1989}}; {\cite[Theorem~A.11]{DonaldsonLehmann2024}}]
	\label{lem:property-1-Raymond}
	Assume that $\|\phi(u_0)\|_{2d}\le \theta_0^{-4}$. Then for all $k\ge 0$,
	\[
	|u_k-u_0|_{3d}<\delta
	\qquad\text{and}\qquad
	\|\phi(u_k)\|_{2d}\le \theta_k^{-4}.
	\]
\end{lemma}

\begin{proof}
		We argue by induction on $k$. The case $k=0$ is immediate. Fix $k\in\mathbb{N}$ and assume that
	\[
	|u_j-u_0|_{3d}<\delta
	\quad\text{and}\quad
	\|\phi(u_j)\|_{2d}\le \theta_j^{-4}
	\qquad\text{for all }0\le j\le k.
	\]
	
	\smallskip
	\noindent\textit{Step 1: The iterates stay in the tame neighbourhood.}
	Since $u_{j+1}=u_j+S_{\theta_j}v_j$, we have
	\[
	u_k+tS_{\theta_k}v_k-u_0=\sum_{j=0}^{k-1}S_{\theta_j}v_j+tS_{\theta_k}v_k,
	\qquad t\in[0,1].
	\]
	Using \eqref{eq:estimates-smoothing-operators} with $(s,t)=(3d,3d)$, the monotonicity
	$|v_j|_{3d}\le |v_j|_{3d+3}$, and Lemma~\ref{lem:property-2-Raymond}, we obtain
	\[
	|u_k+tS_{\theta_k}v_k-u_0|_{3d}
	\le \widehat C_{3d,3d}\sum_{j=0}^{k}|v_j|_{3d}
	\le \widehat C_{3d,3d}\sum_{j=0}^{k}|v_j|_{3d+3}
	\le \widehat C_{3d,3d}\,V\sum_{j=0}^{k}\theta_j^{-3}.
	\]
	By the choice of $(\theta_j)$, $\sum_{j\ge 0}\theta_j^{-3}<\theta_0^{-1}$; hence if
	$\theta_0\ge \widehat C_{3d,3d}\,V/\delta$, then
	\[
	|u_k+tS_{\theta_k}v_k-u_0|_{3d}<\delta
	\qquad\text{for all }t\in[0,1],
	\]
	and in particular $|u_{k+1}-u_0|_{3d}<\delta$.
	\\

	\smallskip
	\noindent\textit{Step 2: Decay of $\phi(u_k)$.}
	Taylor's formula with integral remainder gives
	\[
	\phi(u_{k+1})
	=\phi(u_k)+\phi'(u_k)S_{\theta_k}v_k
	+\int_0^1(1-t)\,\phi''(u_k+tS_{\theta_k}v_k)\bigl(S_{\theta_k}v_k,S_{\theta_k}v_k\bigr)\,dt.
	\]
	Since $v_k=-\psi(u_k)\phi(u_k)$ and \eqref{eq:definition-inverse} implies
	$\phi'(u_k)v_k=-\phi(u_k)-\mathcal E(u_k)(\phi(u_k),\phi(u_k))$,
	we can write $\phi(u_{k+1})=\varphi_1+\varphi_2+\varphi_3$ with
	\begin{align*}
	\varphi_1:=&\phi'(u_k)(S_{\theta_k}v_k-v_k), \\
	\varphi_2:=&\int_0^1(1-t)\,\phi''(\cdots)(S_{\theta_k}v_k,S_{\theta_k}v_k)\,dt,\\
	\varphi_3:=&-\mathcal E(u_k)(\phi(u_k),\phi(u_k)).	
	\end{align*}

	For $\varphi_1$, using \eqref{eq:tame-estimates}, \eqref{eq:estimates-smoothing-operators} (with $s=3d$, $t=3d+3$), and Lemma~\ref{lem:property-2-Raymond},
	\begin{equation}\label{eq:estimate-varphi-1}
		\|\varphi_1\|_{2d}
		\le C_1|S_{\theta_k}v_k-v_k|_{3d}
		\le C_1\widehat C_{3d,3d+3}\theta_k^{-3}|v_k|_{3d+3}
		\le C_1\widehat C_{3d,3d+3}V\,\theta_k^{-6}.
	\end{equation}
	For $\varphi_2$, \eqref{eq:tame-estimates} and \eqref{eq:estimates-smoothing-operators} (with $s=t=3d$) yield
	\begin{equation}\label{eq:estimate-varphi-2}
		\|\varphi_2\|_{2d}
		\le C_2|S_{\theta_k}v_k|_{3d}^2
		\le C_2\widehat C_{3d,3d}^2|v_k|_{3d}^2
		\le C_2\widehat C_{3d,3d}^2V^2\,\theta_k^{-6}.
	\end{equation}
	
	To estimate $\varphi_3$, we first bound $\|\phi(u_k)\|_{3d}$.
	Set $\tau:=3d+d(N+3)$ and split
	$\phi(u_k)=S_{\theta_k^{1/d}}\phi(u_k)+(1-S_{\theta_k^{1/d}})\phi(u_k)$.
	Then \eqref{eq:estimates-smoothing-operators} gives
	\[
	\|\phi(u_k)\|_{3d}
	\le \widehat C_{3d,2d}\theta_k\|\phi(u_k)\|_{2d}
	+\widehat C_{3d,\tau}\theta_k^{-N-3}\|\phi(u_k)\|_{\tau}.
	\]
	Using the inductive hypothesis $\|\phi(u_k)\|_{2d}\le \theta_k^{-4}$, the tame bound
	$\|\phi(u_k)\|_{\tau}\le C_\tau(1+|u_k|_{\tau+d})$, and the monotonicity of
	$(1+|u_k|_{\tau+d})\theta_k^{-N}$ (which holds provided $\theta_0\ge U_{\tau-d}$),
	we obtain
	\[
	\|\phi(u_k)\|_{3d}
	\le \Bigl(\widehat C_{3d,2d}+\widehat C_{3d,\tau}C_\tau(1+|u_0|_{\tau+d})\Bigr)\theta_k^{-3}.
	\]
	Hence, by \eqref{eq:estimate-qadratic-error},
	\begin{equation}\label{eq:estimate-varphi-3}
		\|\varphi_3\|_{2d}
		\le \widetilde C_3\|\phi(u_k)\|_{3d}^2
		\le \widetilde C_3\Bigl(\widehat C_{3d,2d}+\widehat C_{3d,\tau}C_\tau(1+|u_0|_{\tau+d})\Bigr)^2\theta_k^{-6}.
	\end{equation}
	
	Combining \eqref{eq:estimate-varphi-1}--\eqref{eq:estimate-varphi-3} and the definition of $C_0$,
	we obtain $\|\phi(u_{k+1})\|_{2d}\le C_0\theta_k^{-6}$.
	Since $\theta_{k+1}=\theta_k^{5/4}$, we have $\theta_k^{-6}=\theta_k^{-1}\theta_{k+1}^{-4}$. Therefore,
	\[
	\|\phi(u_{k+1})\|_{2d}\le (C_0\theta_k^{-1})\,\theta_{k+1}^{-4}\le \theta_{k+1}^{-4},
	\]
	provided $\theta_0\ge C_0$. This closes the induction.
\end{proof}

\medskip
\noindent
For the remainder of the argument we refer to \cite{Raymond1989}. In particular, \cite[Lemma~2]{Raymond1989}
provides a strengthened version of Lemma~\ref{lem:property-2-Raymond}: there exists a family of constants
$\{V_s\}_{s\ge d}\subset(0,\infty)$ such that
\[
|v_k|_s \le V_s\,\theta_k^{-3}\qquad\text{for all }k\ge 0.
\]
The explicit values of $V_s$ are not needed for Proposition~\ref{prop:explicit-nash-moser}, and hence the proof can be copied directly.
Finally, for $s\ge d$,
\[
|u_{k+1}-u_k|_s = |S_{\theta_k}v_k|_s \le \widehat C_{s,s}\,V_s\,\theta_k^{-3},
\]
so $(u_k)$ is Cauchy in each $|\cdot|_s$ and converges to some $u\in\mathcal F$ with $\phi(u)=0$ by continuity.

\bibliography{references.bib}
\bibliographystyle{amsplain}

\end{document}